\documentclass{amsart}

\usepackage{geometry}
\usepackage{amsmath}
\usepackage{amssymb}
\usepackage{amsthm}
\usepackage{mathrsfs}
\usepackage{tikz-cd}
\usepackage{hyperref}
\usepackage{cleveref}
\usepackage{cite}
\usepackage{enumitem}
\setlist[enumerate]{label=(\arabic*)}

\newcommand{\R}{\mathbb{R}}
\newcommand{\Z}{\mathbb{Z}}

\newcommand{\C}{\mathcal{C}}
\newcommand{\N}{\mathbb{N}}

\newcommand{\X}{\mathbf{X}}
\newcommand{\Y}{\mathbf{Y}}
\newcommand{\D}{\mathcal{D}}

\DeclareMathOperator{\aut}{Aut}
\DeclareMathOperator{\dist}{dist}
\DeclareMathOperator{\id}{id}
\DeclareMathOperator{\diam}{diam}
\DeclareMathOperator{\prob}{Prob}
\DeclareMathOperator{\cov}{cov}
\DeclareMathOperator{\co}{Co}
\DeclareMathOperator{\sym}{Sym}

\newcommand{\la}{\lambda}
\renewcommand{\a}{\alpha}
\renewcommand{\b}{\beta}
\renewcommand{\phi}{\varphi}

\newcommand{\beq}{\begin{equation}}
\newcommand{\eeq}{\end{equation}}

\makeatletter
\newtheorem*{rep@theorem}{\rep@title}
\newcommand{\newreptheorem}[2]{
    \newenvironment{rep#1}[1]{
        \def\rep@title{#2 \ref{##1}}
        \begin{rep@theorem}
    }
    {\end{rep@theorem}}
}
\makeatother

\newtheorem{theorem}{Theorem}[section]
\newtheorem{corollary}[theorem]{Corollary}
\newtheorem{lemma}[theorem]{Lemma}
\newtheorem{proposition}[theorem]{Proposition}
\newreptheorem{theorem}{Theorem}
\newreptheorem{corollary}{Corollary}

\theoremstyle{definition}
\newtheorem{definition}[theorem]{Definition}
\newtheorem{remark}[theorem]{Remark}

\newcommand{\abs}[1]{\left\lvert #1 \right\rvert}
\newcommand{\ab}[1]{\left\langle #1 \right\rangle}
\renewcommand{\bar}[1]{\overline{#1}}
\newcommand{\susbeteq}{\subseteq}
\newcommand{\wt}{\widetilde}

\newcommand{\h}{h_{\operatorname{slow}}}

\title{Relative slow entropy}
\author{Adam Lott}
\date{March 14, 2023}
\address{Department of Mathematics, University of California, Los Angeles, Los Angeles, CA 90095}
\email{\href{mailto:adamlott99@math.ucla.edu}{adamlott99@math.ucla.edu}}

\begin{document}

\begin{abstract}
In 1997, Katok--Thouvenot and Ferenczi independently introduced a notion of ``slow entropy'' as a way to quantitatively compare measure-preserving systems with zero entropy.
We develop a relative version of this theory for a measure-preserving system conditioned on a given factor.
Our new definition inherits many desirable properties that make it a natural generalization of both the Katok--Thouvenot/Ferenczi theory and the classical conditional Kolmogorov--Sinai entropy.
As an application, we prove a relative version of a result of Ferenczi that classifies isometric systems in terms of their slow entropy.
We also introduce a new definition for the notion of a rigid extension and investigate its relationship to relative slow entropy.
\end{abstract}

\maketitle{}

\section{Introduction}

\subsection{Background}

In the study of measure-preserving systems, one of the most powerful and classical isomorphism invariants is the Kolmogorov--Sinai entropy rate.
It was first introduced for actions of $\Z$ by Kolmogorov and Sinai in order to answer the question of whether all Bernoulli shifts are isomorphic \cite{kolmogorov1958entropy, kolmogorov1959entropy,sinai1959entropy1, sinai1959entropy2}.
The theory was eventually extended to actions of any amenable group by Kieffer, Ornstein, Weiss, and others (see e.g. \cite{katznelson1972commuting, kieffer1975amenable, ornstein1980rokhlin, ollagnier1985book}), and has been applied to many other problems in ergodic theory.
Much of the theory has also been generalized to describe the \emph{conditional} entropy of an action relative to a given factor \cite{abramov1962skew, ward1992amenable, rudolph2000mixing}.

Given a measure space $(X, \mu)$ and a measure-preserving action $T$, the entropy rate is usually defined as a limit of the normalized Shannon entropies of sequences of partitions of $X$.
However, it can be equivalently described as the exponential growth rate of the number of quasi-orbits of $T$ required to cover ``most'' of $X$ (according to the measure $\mu$).
This can be viewed as a measure-theoretic analogue of Bowen's definition of topological entropy \cite{bowen1971entropy} (see also \cite[section 7.2]{walters1982book}).

More recently, Katok--Thouvenot and Ferenczi observed independently \cite{katok1997slow, ferenczi1997complexity} that for systems with zero entropy, while the \emph{exponential} growth rate of the number of quasi-orbits is always zero, one can look at the growth rate with respect to slower rate functions and still extract a useful isomorphism invariant.
The usefulness of this notion is that it provides a quantitative way to compare two different zero entropy systems.
Katok and Thouvenot used the phrase ``slow entropy'' to describe this family of invariants and they used it to show that certain kinds of measure-preserving systems have no smooth realizations.
One of Ferenczi's original uses for it was to give a characterization of Kronecker systems, and more recently there have been several results connecting slow entropy to other dynamical properties \cite{kanigowski2019parabolic, cyr2020subshift,adams2021generic,austin2021dominant,lott2022dominant}.

The aim of this paper is to develop a conditional version of the slow entropy theory.
The notion of relative slow entropy that we will present satisfies many properties that make it a natural generalization of both the relative Kolmogorov--Sinai entropy and the absolute Katok--Thouvenot slow entropy.
As an application, we also include some examples of natural dynamical properties that can be characterized using relative slow entropy.

\subsection{Definitions and notation}

Let $G$ be a countable discrete amenable group.
A (left) {\bf F\o lner sequence} for $G$ is a sequence of finite sets $F_n \subseteq G$ satisfying
\[
    \lim_{n \to \infty} \frac{|g F_n \cap F_n|}{|F_n|} \ = \ 1  
\]
for any fixed $g \in G$.
Unless otherwise specified, all F\o lner sequences will be assumed to be left F\o lner sequences.

Let $T$ be a measure preserving action of $G$ on the standard measure space $(X, \mathcal{B}_X, \mu)$ and write $\X = (X, \mathcal{B}_X, \mu,T)$.
For a finite partition $P = \{P_0, \dots, P_{r-1}\}$ of $X$, define $P(x)$ to be the unique $i \in \{0, 1, \dots, r-1\}$ such that $x \in P_i$.
Also, for a finite subset $F \subseteq G$, let $P^F$ be the partition $\bigvee_{f \in F} T^{f^{-1}} P$ and let $P^F(x)$ be the {\bf $\mathbf{(P,F)}$-name of $\mathbf{x}$}, i.e. the word 
\[
(P(T^f x))_{f \in F} \ \in \ \{0,1,\dots, r-1\}^F.
\]
For any partition $P$ and finite subset $F \subseteq G$, define the pseudo-metric $d_{P,F}$ on $X$ by
\[
d_{P,F}(x,x') \ = \ \frac{1}{|F|} \sum_{f \in F} 1_{P(T^f x) \neq P(T^f x')}.
\]
This is just the normalized Hamming distance between the two names $P^F(x)$ and $P^F(x')$.
For $E \subseteq X$, let $\diam_{P,F}(E)$ be the diameter of $E$ with respect to $d_{P,F}$.

\begin{definition}
    Let $F$ be a finite subset of $G$, $\la$ any probability measure on $X$, $P$ any partition of $X$, and $\epsilon > 0$.
    We denote by 
    \[
    \cov(\la, P, F, \epsilon)
    \]
    the \textbf{Hamming $\epsilon$-covering number} -- the smallest $M$ such that there exist sets $E_1, \dots, E_M \subseteq X$ satisfying $\diam_{P,F}(E_i) \leq \epsilon$ for all $i$ and $\la \left( \bigcup E_i \right) \geq 1-\epsilon$.
    We remark that this number is always finite because $X$ is totally bounded when equipped with any of the pseudo-metrics $d_{P,F}$ (in fact, for any $P$ and any $F$, $X$ is the union of finitely many sets of diameter $0$ according to $d_{P,F}$).
\end{definition}

\begin{definition}
    A {\bf rate function} is an increasing function $U: \N \to (0,\infty)$ such that $U(n) \to \infty$ as $n \to \infty$.
\end{definition}

In \cite{katok1997slow} and \cite{ferenczi1997complexity}, slow entropy is defined as follows.
Let $U$ be a rate function and $(F_n)$ be a F\o lner sequence for $G$.
Given a partition $P$, one defines
\[
\h^{U, (F_n)}(\X, P) \ = \ \sup_{\epsilon > 0} \ \limsup_{n \to \infty} \frac{\cov(\mu, P, F_n, \epsilon)}{U(|F_n|)}.
\]
Then the slow entropy of 
$\X$
with respect to the rate function $U$ and F\o lner sequence $(F_n)$ is defined to be
\[
\h^{U, (F_n)}(\X) \ = \ \sup_{P} \ \h^{U, (F_n)}(\X, P),
\]
where the supremum is taken over all finite partitions of $X$ into measurable sets.

We now relativize this definition to an extension 
$\pi: \X \to \Y := (Y, \mathcal{B}_Y, \nu, S)$.
Let $\mu = \int \mu_y \,d\nu(y)$ be the disintegration of $\mu$ over $\pi$ (see for example \cite[Theorem 33.3]{billingsley1995book}).

\begin{definition} \label{definition: RelativeCoveringNumber}
Given a partition $P$, a finite set $F \subseteq G$, and $\epsilon > 0$, we define the {\bf relative Hamming $\epsilon$-covering number} $\cov(\mu, P, F, \epsilon \,|\, \pi)$ to be the smallest $M$ with the following property: there exists a set $S \subseteq Y$ with $\nu(S) \geq 1-\epsilon$ such that for any $y \in S$, $\cov \left( \mu_y, P, F, \epsilon \right) \leq M$.
It can equivalently be defined as
\[
\cov(\mu, P, F, \epsilon \,|\, \pi) \ = \ \inf_{S \subseteq Y \,:\, \nu(S) \geq 1-\epsilon} \ \ \sup_{y \in S} \ \cov(\mu_y, P, F, \epsilon).
\]
\end{definition}

\begin{definition}
We now define the {\bf relative slow entropy} of $\pi$ by the analogous formulas
\begin{align*}
\h^{U, (F_n)}(\X, P \,|\, \pi) \ &= \ \sup_{\epsilon > 0} \ \limsup_{n \to \infty} \frac{\cov(\mu, P, F_n, \epsilon \,| \, \pi)}{U(|F_n|)} \\
\h^{U, (F_n)}(\X \,|\, \pi) \ &= \ \sup_{P} \ \h^{U, (F_n)}(\X, P \,|\, \pi).
\end{align*}
\end{definition}

\subsection{Outline of results}

\Cref{sec: BasicProperties} is devoted to establishing the basic properties of relative slow entropy.
The first essential property it is that it is monotone in both the top system and bottom system.

\begin{reptheorem}{theorem: MonotoneUnderExtensions}
    Consider a commutative diagram of the form
    \begin{center}
    \begin{tikzcd}
        \X \arrow[dr, "\bar{\phi}"] \arrow[dd, "\pi"] & \\
        & \X' \arrow[d, "\pi'"] \\
        \Y \arrow[r, "\phi"', "\sim"] & \Y'
    \end{tikzcd}
    \end{center}
    where $\bar{\phi}$ is a factor map and $\phi$ is an isomorphism.
    Then 
    \[
        \h^{U, (F_n)}(\X' \,|\, \pi') \ \leq \ \h^{U, (F_n)}(\X \,|\, \pi)
    \]
    for any rate function $U$ and any F\o lner sequence $(F_n)$.
\end{reptheorem}

\begin{reptheorem}{theorem: MonotoneUnderMoreConditioning}
    Consider a commutative diagram of the form
    \begin{center}
    \begin{tikzcd}
        \X \arrow[r, "\bar{\phi}"', "\sim"] \arrow[d, "\pi"] & \X' \arrow[dd, "\pi'"] \\
        \Y \arrow[dr, "\phi"] &  \\
        & \Y' \\
    \end{tikzcd}
    \end{center}
    where $\bar{\phi}$ is an isomorphism and $\phi$ is a factor map.
    Then
    \[
        \h^{U, (F_n)}(\X \,|\, \pi) \ \leq \ \h^{U, (F_n)}(\X' \,|\, \pi')
    \]  
    for any rate function $U$ and any F\o lner sequence $(F_n)$.
\end{reptheorem}

We also note that it follows immediately from either one of these two results that if $\bar{\phi}$ and $\phi$ are both isomorphisms, then $\h^{U, (F_n)}(\X \,|\, \pi) = \h^{U, (F_n)}(\X' \,|\, \pi')$ for any choice of $U$ and $(F_n)$.
We refer to this fact throughout as the ``isomorphism invariance of relative slow entropy''.

\begin{remark}
Let 
$\pi: \X \to \Y$ and $\pi': \X' \to \Y'$
be two extensions.
If $U$ is any rate function and $(F_n)$ is any F\o lner sequence, then in order to show $\h^{U, (F_n)}(\X \,|\, \pi) \geq \h^{U, (F_n)}(\X' \,|\, \pi')$ it is sufficient to show that for any $\epsilon' > 0$ and partition $P'$ of $X'$, there exist an $\epsilon > 0$ and partition $P$ of $X$ so that
\[
\cov(\mu, P, F_n, \epsilon \,|\, \pi) \ \geq \ \cov(\mu', P', F_n, \epsilon' \,|\, \pi')
\]
for all $n$ sufficiently large.
\end{remark}

One of the most important properties of the classical Kolmogorov-Sinai relative entropy rate is that it can be computed via a sequence of relatively generating partitions \cite[Theorem 2.20]{einsiedler2021entropy} (see \Cref{sec: BasicProperties} for definition).
We show that relative slow entropy also has this property.

\begin{reptheorem}{theorem: GeneratingPartitionsDominate}
    Let $(P_m)_{m=1}^{\infty}$ be a sequence of refining partitions that is generating for $\X$ relative to $\pi$.
    Then 
    \[
    \h^{U, (F_n)}(\X \,|\, \pi) \ = \ \lim_{m \to \infty} \ \h^{U, (F_n)}(\X, P_m \,|\, \pi) \ = \ \sup_{m} \ \h^{U, (F_n)}(\X, P_m \,|\, \pi) 
    \]
    for any rate function $U$ and any F\o lner sequence $(F_n)$.
\end{reptheorem}

In the non-relative setting for slow entropy, it is known (see \cite[Proposition 2]{ferenczi1997complexity} or \cite[Theorem 1.1]{katok1980lyapunov}) that using an exponential rate function recovers the classical entropy rate.
We show the same fact for relative slow entropy.

\begin{reptheorem}{theorem: RecoverKSEntropy}
    Assume that $\X$ is ergodic.
    For $t > 0$, let $U_t(n) = \exp(t \cdot n)$.
    Let $(F_n)$ be any F\o lner sequence.
    Then we have
    \[
        \h^{U_t, (F_n)}(\X, P \,|\, \pi) \ = \ \begin{cases} \infty & \text{if} \ \ t < h_{\operatorname{KS}}(\X, P \,|\, \pi) \\ 0 & \text{if} \ \ t > h_{\operatorname{KS}}(\X, P \,|\, \pi) \end{cases}
    \] 
    for any partition $P$.
    Equivalently, we have
    \[
    \sup_{\epsilon > 0} \ \limsup_{n \to \infty} \frac{\log \cov(\mu, P, F_n, \epsilon \,|\, \pi)}{|F_n|} \ = \ 
    h_{\operatorname{KS}}(\X, P \,|\, \pi)
    \]
    for any partition $P$.
\end{reptheorem}

In \Cref{sec: Isometric}, we characterize isometric and weakly mixing extensions in terms of their relative slow entropy.
These results are both relativizations and generalizations to all amenable groups of Ferenczi's result \cite[Proposition 3]{ferenczi1997complexity}.

\begin{reptheorem}{theorem: IsometricEquivBounded}
    Suppose that 
    $\X$
    is ergodic.
    Then the following are equivalent.
    \begin{enumerate}
        \item $\pi$ is an isometric extension.
        \item There exists a F\o lner sequence $(F_n)$ such that 
        $\h^{U, (F_n)}(\X \,|\, \pi) = 0$ 
        for all rate functions $U$.
        \item For any F\o lner sequence $(F_n)$,
        $\h^{U, (F_n)}(\X \,|\, \pi) = 0$
         for all rate functions $U$.
    \end{enumerate}
\end{reptheorem}

\begin{repcorollary}{corollary: WeakMixingCriterion}
    Suppose that 
    $\X$
    is ergodic.
    Then $\pi$ is a weakly mixing extension if and only if for every partition $P$ that is not $\Y$-measurable, there exists a rate function $U$ and a F\o lner sequence $(F_n)$ such that 
    $\h^{U, (F_n)}(\X, P\,|\, \pi) > 0$.
\end{repcorollary}

Finally, in \Cref{sec: Rigidity}, we specialize to $G = \Z$ and explore the notion of relative rigidity.
Rigidity is a classical dynamical property that has been well studied, but there is no standardized relative version of the theory.
We propose a new definition of what it means for an extension $\pi$ to be rigid and investigate some of its consequences.
First, we show that rigid extensions are generic.

\begin{reptheorem}{theorem: RigidGeneric}
    Let $\Y$ be ergodic.
    Then the generic extension $\pi: \X \to \Y$ is rigid.
\end{reptheorem}

We also obtain a sufficient condition in terms of relative slow entropy.

\begin{reptheorem}{theorem: RelativeRigidtySufficientCondition}
    Let $\Y$ be ergodic.
    Suppose that there exists a F\o lner sequence $(F_n)$ for $\N$ such that $\h^{L, (F_n)}(\X \,|\, \pi) = 0$, where $L(n) = \log n$.
    Then $\pi$ is a rigid extension.
\end{reptheorem}

\noindent In the non-relative setting, we are also able to give full characterizations of rigidity and mild mixing in terms of slow entropy, which may be of independent interest.

\begin{reptheorem}{theorem: RigidtyCondition}
    The following are equivalent.
    \begin{enumerate}
        \item 
        $\X$
        is rigid.
        \item For every rate function $U$, there exists a F\o lner sequence $(F_n)$ for $\N$ such that $\h^{U, (F_n)}(\X) = 0$.
        \item For the rate function $L(n) = \log n$, there exists a F\o lner sequence $(F_n)$ for $\N$ such that $\h^{L, (F_n)}(\X) = 0$.
    \end{enumerate} 
\end{reptheorem}

\begin{repcorollary}{corollary: MildMixingCondition}
    A measure preserving system
    $\X$
    is mildly mixing if and only if for all partitions $P$ of $X$ and all F\o lner sequences $(F_n)$ for $\N$, we have $\h^{L, (F_n)}(\X, P) > 0$, where $L(n) = \log n$.
\end{repcorollary}
    
\subsection{Acknowledgements}

I am grateful to Bryna Kra for introducing me to the work \cite{ferenczi1997complexity}, which was the inspiration for \cref{sec: Isometric}.
I also thank Tim Austin for his constant advice and guidance and James Leng and Benjy Weiss for helpful conversations.

This work was partially supported by NSF grant DMS-1855694.

\section{Basic properties} \label{sec: BasicProperties}

\subsection{Monotonicity and isomorphism invariance}

Recall that we write $\mu = \int \mu_y \,d\nu(y)$ for the disintegration of $\mu$ over the map $\pi$.

\begin{definition} \label{definition: IsomorphicExtensions}
    Let $\pi: \X \to \Y$ and $\pi': \X' \to \Y'$ be extensions.  We say that $\pi$ is an \textbf{upward extension} of $\pi'$ if there is a commutative diagram
    \begin{center}
    \begin{tikzcd}
        \X \arrow[dr, "\bar{\phi}"] \arrow[dd, "\pi"] & \\
        & \X' \arrow[d, "\pi'"] \\
        \Y \arrow[r, "\phi"', "\sim"] & \Y'
    \end{tikzcd}
    \end{center}
    where $\bar{\phi}$ is a factor map and $\phi$ is an isomorphism.
\end{definition}


\begin{lemma} 
\label{lemma: PushForwardConditionalMeasures}
    Let $\pi$ be an upward extension of $\pi'$ as in \Cref{definition: IsomorphicExtensions}.
    Then 
    \[
        \bar{\phi}_* \left( \mu_{\phi^{-1} y'} \right) \ = \ \mu'_{y'}
    \]
    for $\nu'$-a.e. $y' \in Y'$.
\end{lemma}

\begin{proof}
By the essential uniqueness of disintegrations, it suffices to show the two properties
\begin{enumerate}
    \item $\mu' = \int \bar{\phi}_* \left( \mu_{\phi^{-1} y'} \right) \,d\nu'(y')$, and

    \item $\bar{\phi}_* \left( \mu_{\phi^{-1} y'} \right)((\pi')^{-1} y') = 1$ for $\nu'$-a.e. $y'$.
\end{enumerate}
Property (1) is immediate from the definition of a factor map.
To show (2), it is sufficient to show that
$\int \bar{\phi}_* \left( \mu_{\phi^{-1} y'} \right)((\pi')^{-1} y') \,d\nu'(y') = 1$.
This is also immediate from the fact that the diagram in \Cref{definition: IsomorphicExtensions} commutes.
\end{proof}

\begin{theorem} \label{theorem: MonotoneUnderExtensions}
    Let $\pi$ be an upward extension of $\pi'$. 
    Then 
    \[
        \h^{U, (F_n)}(\X' \,|\, \pi') \ \leq \ \h^{U, (F_n)}(\X \,|\, \pi)
    \]
    for any rate function $U$ and any F\o lner sequence $(F_n)$.
\end{theorem}

\begin{proof}
Let $\phi: \Y \to \Y'$ and $\bar{\phi}: \X \to \X'$ be the maps as in \Cref{definition: IsomorphicExtensions}.
Let $\epsilon < 0$ and let $P'$ be any finite partition of $X'$.
Let $P$ be the partition $\bar{\phi}^{-1} P'$ of $X$.
Suppose that $\cov(\mu, P, F_n, \epsilon \,|\, \pi ) = L$.
Let $S \subseteq Y$ be such that $\nu(Y) \geq 1-\epsilon$ and $\cov(\mu_y, P, F_n, \epsilon) \leq L$ for all $y \in S$.
Fix such a $y$.
We will show that $\cov(\mu'_{\phi y}, P', F_n, \epsilon) \leq L$ as well.

Let $B_1, \dots, B_L \subseteq X$ be such that $\diam_{P,F_n}(B_i) \leq \epsilon$ and $\mu_y \left( \bigcup B_i \right) \geq 1-\epsilon$.

We want to claim that the sets $\bar{\phi}B_1, \dots, \bar{\phi}B_L$ satisfy the analogous properties in $X'$, but because $\bar{\phi}$ is only a factor map and not necessarily an isomorphism, these sets need not be measurable.
So we define $B_i'$ to be the union of all $(P')^{F_n}$-cells that meet $\bar{\phi}B_i$, and we show that the sets $B_1', \dots, B_L'$ have the right properties.

To check the diameter condition, note that by construction, $\diam_{P', F_n}B_i' = \diam_{P', F_n}(\bar{\phi}B_i)$, so it suffices to estimate the latter.
Observe that if $\bar{\phi}x, \bar{\phi}z \in \bar{\phi}B_i$, then
\begin{align*}
    d_{P',F_n}(\bar{\phi}x, \bar{\phi}z) \ &= \ \frac{1}{|F_n|} \sum_{f \in F_n} 1_{P'(T'^f \bar{\phi} x) \neq P'(T'^f \bar{\phi} z)} \ = \ \frac{1}{|F_n|} \sum_{f \in F_n} 1_{P'(\bar{\phi} T^f x) \neq P'(\bar{\phi} T^f z)} \\
    &= \ \frac{1}{|F_n|} \sum_{f \in F_n} 1_{P(T^f x) \neq P(T^f z)} \ = \ d_{P,F_n}(x,z) \ \leq \ \epsilon,
\end{align*}
    so $\diam_{P',F_n}(\bar{\phi}B_i) \leq \epsilon$.

Now to check the measure condition, apply \Cref{lemma: PushForwardConditionalMeasures} to write
\[
    \mu'_{\phi y} \left( \bigcup B_i' \right) \ = \ \mu_y \left( \bar{\phi}^{-1} \bigcup B_i' \right) \ = \ \mu_y \left( \bigcup \bar{\phi}^{-1} B_i' \right) \ \geq \ \mu_y \left( \bigcup B_i \right) \ \geq \ 1-\epsilon.
\]

This shows that $\cov(\mu'_{\phi y}, P', F_n, \epsilon) \leq L$ for all $y \in S$.
Taking $S' = \phi(S)$, we have $\nu'(S') \geq 1-\epsilon$ as well, so $\cov(\mu', P', F_n, \epsilon \,|\, \pi) \leq L = \cov(\mu, P, F_n, \epsilon \,|\, \pi)$ for all $n$.
Since this holds for any $P'$ and any $\epsilon$, the desired result follows.
\end{proof}


\begin{definition}
\label{definition: DownwardExtension}
    Let $\pi: \X \to \Y$ and $\pi' : \X' \to \Y'$ be extensions.
    We say that $\pi'$ is a \textbf{downward extension} of $\pi$ if there is a commutative diagram of the form
    \begin{center}
    \begin{tikzcd}
    \X \arrow[r, "\bar{\phi}"', "\sim"] \arrow[d, "\pi"] & \X' \arrow[dd, "\pi'"] \\
    \Y \arrow[dr, "\phi"] &  \\
    & \Y' \\
    \end{tikzcd}
    \end{center}
    where $\bar{\phi}$ is an isomorphism and $\phi$ is a factor map.
\end{definition}

\begin{lemma}
\label{lemma: PullBackConditionalMeasures}
    Let $\pi'$ be a downward extension of $\pi$ as in \Cref{definition: DownwardExtension}.
    Write $\nu = \int \nu_{y'} \,d\nu'(y')$ for the disintegration of $\nu$ over $\phi$.
    Then we have
    \[
        \mu'_{y'} \ = \ \bar{\phi}_* \int \mu_y \,d\nu_{y'}(y)
    \]  
    for $\nu'$-a.e. $y' \in Y'$.
\end{lemma}

\begin{proof}
    Denote the right hand side of the above by $\la_{y'}$.
    By the essential uniqueness of disintegration, it suffices to show the two properties
    \begin{enumerate}
        \item $\mu' = \int \la_{y'} \,d\nu'(y')$, and
        \item for $\nu'$-a.e. $y'$, $\la_{y'}$ is supported on the fiber $(\pi')^{-1}y'$.
    \end{enumerate}
    Property (1) follows from the fact that $\iint \mu_y \,d\nu_{y'}(y) \,d\nu(y) = \int \mu_y \,d\nu(y) = \mu$.
    To see property (2), note that $\int \mu_y \,d\nu_{y'}(y)$ is a mixture of measures that are all supported on $\pi^{-1} \phi^{-1} y' = \bar{\phi}^{-1} (\pi')^{-1} y'$.
    Therefore $\la_{y'}$ is supported on $(\pi')^{-1} y'$ as desired. 
\end{proof}

\begin{theorem}
\renewcommand{\Z}{\mathbf{Z}}
\label{theorem: MonotoneUnderMoreConditioning}
    Suppose that $\pi'$ is a downward extension of $\pi$.
    Then
    \[
        \h^{U, (F_n)}(\X \,|\, \pi) \ \leq \ \h^{U, (F_n)}(\X' \,|\, \pi')
    \]  
    for any rate function $U$ and any F\o lner sequence $(F_n)$.
\end{theorem}

\begin{proof}
    Let $\epsilon > 0$ and let $P$ be any partition of $X$, and set $P' = \bar{\phi} P$.
    Suppose that 
    \[
    \cov(\mu', P', F_n, \epsilon^2/4 \,|\, \pi') \ = \ L.
    \]
    It suffices to show that $\cov(\mu, P, F_n, \epsilon \,|\, \pi) \leq L$ as well.

    Let $S' \subseteq Y'$ be the set of $y' \in Y'$ that satisfy $\cov(\mu'_{y'}, P', F_n, \epsilon^2/4) \leq L$ and note that by definition, $\nu'(S') \geq 1-\epsilon^2/4$.
    Fix $y' \in S'$.
    Let $B_1', \dots, B_L'$ be subsets of $X'$ satisfying $\diam_{P', F_n}(B_i') \leq \epsilon^2/4$ and $\mu'_{y'} \left( \bigcup B_i' \right) \geq 1-\epsilon^2/4$.
    Set $B_i = \bar{\phi}^{-1} B_i'$.
    First, note that it follows immediately from the fact that $\bar{\phi}$ is an isomorphism that $\diam_{P,F_n}(B_i) = \diam_{P', F_n}(B_i') \leq \epsilon^2/4$.

    We now show that for most $y$, the sets $B_i$ cover most of $\mu_y$.
    By \Cref{lemma: PullBackConditionalMeasures}, we have
    \[
        1 - \epsilon^2/4 \ \leq \ \mu'_{y'} \left( \bigcup B_i' \right) \ = \ \int \mu_y \left( \bar{\phi}^{-1} \bigcup B_i' \right) \,d\nu_{y'}(y) \ = \ \int \mu_y \left( \bigcup B_i \right) \,d\nu_{y'}(y),
    \]
    so Markov's inequality implies that the set 
    \[
        S(y') \ := \ \left\{ y \in \phi^{-1} y' : \mu_y \left( \bigcup B_i \right) \geq 1-\epsilon/2 \right\}
    \]     
    satisfies $\nu_{y'}(S(y')) \geq 1-\epsilon/2$.
    This shows that any $y \in S(y')$ satisfies 
    \[
        \cov(\mu_y, P, F_n, \epsilon) \ \leq \ L.
    \]
    
    Finally, this construction was valid for any $y' \in S'$.
    Therefore, let
    \[
        S \ := \ \left\{ y \in Y : \cov(\mu_y, P, F_n, \epsilon) \leq L \right\}
    \]
    and observe that $S$ contains $S(y')$ for all $y' \in S'$.
    So we conclude that
    \begin{align*}
        \nu(S) \ &= \ \int \nu_{y'}(S) \,d\nu'(y') \ \geq \ \int_{y' \in S'} \nu_{y'}(S) \,d\nu'(y') \ \geq \ \int_{y' \in S'} \nu_{y'}(S(y')) \,d\nu'(y') \\ 
        &\geq \ (1-\epsilon/2)^2 \ \geq \ 1-\epsilon,
    \end{align*}
    implying that $\cov(\mu, P, F_n, \epsilon \,|\, \pi) \leq L$ as desired.
\end{proof}

\subsection{Relatively generating partitions}

\begin{definition}
A partition $P$ of $X$ is said to be \textbf{generating for $\X$ relative to} $\pi$ if
\[
P^G \vee \pi^{-1} \mathcal{B}_Y \ = \ \mathcal{B}_X \mod \mu.
\]
Here we have identified in the natural way the partition $P^G$ with the $\sigma$-algebra that it induces.
Similarly, a sequence of partitions $(P_m)_{m=1}^{\infty}$ is said to be generating for $\X$ relative to $\pi$ if
\[
\left( \bigvee_{m=1}^{\infty} P_m^G \right) \vee \pi^{-1} \mathcal{B}_Y \ = \ \mathcal{B}_X \mod \mu.
\]
The sequence is also said to be \textbf{refining} if $P_{m+1}$ refines $P_m$ for every $m$.
\end{definition}

One of the most important properties of the classical Kolmogorov-Sinai relative entropy is that it can be computed via a sequence of relatively generating partitions \cite[Theorem 2.20]{einsiedler2021entropy}.
In this section, we show an analogous result for relative slow entropy.

\begin{theorem} \label{theorem: GeneratingPartitionsDominate}
Let $(P_m)_{m=1}^{\infty}$ be a refining sequence of partitions that is generating for $\X$ relative to $\pi$.
Then 
\[
\h^{U, (F_n)}(\X \,|\, \pi) \ = \ \lim_{m \to \infty} \h^{U, (F_n)}(\X, P_m \,|\, \pi) \ = \ \sup_{m} \ \h^{U, (F_n)}(\X, P_m \,|\, \pi)
\]
for any rate function $U$ and any F\o lner sequence $(F_n)$.
\end{theorem}

\begin{remark}
    In the case where $\Y$ is trivial, this reduces to the non-relative version of the same result (see \cite[Proposition 1]{katok1997slow} or \cite[Lemma 1]{ferenczi1997complexity}).
\end{remark}

\begin{definition} \label{definition: PartitionDistance}
For finite partitions $P = \{P_1, \dots, P_r\}$ and $Q = \{Q_1, \dots, Q_r\}$ of $X$ and a probability measure $\la \in \prob(X)$, the partition distance with respect to $\la$ is defined as
\[
\dist_{\la}(P,Q) \ = \ \la \{x \in X : P(x) \neq Q(x) \} \ = \ \frac12 \sum_{i=1}^{r} \la(P_i \,\triangle\, Q_i).
\]
\end{definition}

\begin{lemma} \label{lemma: ApproximateByNearbyPartition}
For any partition $Q$ of $X$ and any $0 < \epsilon < 1$, there exists $\epsilon' > 0$ such that if $Q'$ is any other partition of $X$ satisfying $\dist_{\mu}(Q,Q') \leq \epsilon'$, then
\[
\cov(\mu, Q, F_n, \epsilon \,|\, \pi) \ \leq \ \cov(\mu, Q', F_n, \epsilon' \,|\, \pi) \qquad \text{for all $n$ sufficiently large.}
\]
\end{lemma}

\begin{proof}
Set $\epsilon' = (\epsilon/2)^4$.
Suppose that $Q'$ satisfies $\dist_\mu(Q,Q') = \mu\{x : Q(x) \neq Q'(x) \} \leq \epsilon'$ and define
\[
X' \ := \ \left\{ x : \frac{1}{|F_n|} \sum_{f \in F_n} 1_{Q(T^f x) \neq Q'(T^f x)} \leq \sqrt{\epsilon'} \right\}.
\]
By Markov's inequality, we have
\[
\mu(X \setminus X') \ \leq \ \frac{1}{\sqrt{\epsilon'}} \int \frac{1}{|F_n|} \sum_{f \in F_n} 1_{\{z : Q(z) \neq Q'(z)\}}(T^f x) \,d\mu(x) \ \leq \  \frac{1}{\sqrt{\epsilon'}} \dist_\mu(Q,Q') \ \leq \ \sqrt{\epsilon'}.
\]
Applying Markov's inequality again, we obtain a set $S_1 \subseteq Y$ such that $\nu(S_1) \geq 1 - (\epsilon')^{1/4} = 1-\epsilon/2$ such that $\mu_y(X') \geq 1 - (\epsilon')^{1/4} = 1-\epsilon/2$ for all $y \in S_1$.

Now let $L = \cov(\mu, Q', F_n, \epsilon' \,|\, \pi)$ and let $S_2 \subseteq Y$ be such that $\nu(S_2) \geq 1-\epsilon'$ and 
\[
\cov(\mu_y, Q', F_n, \epsilon') \leq L
\]
for $y \in S_2$.
We claim that $\cov(\mu_y, Q, F_n, \epsilon) \leq L$ for $y \in S := S_1 \cap S_2$.
Because $\nu(S) \geq 1 - \epsilon' - \epsilon/2 \geq 1 - \epsilon$, this is enough.
Fix $y \in S$ and let $B_1, \dots, B_L \subset X$ be such that $\diam_{Q', F_n}(B_i) \leq \epsilon'$ and $\mu_y \left( \bigcup B_i \right) \geq 1-\epsilon'$.
Let $B_i' = B_i \cap X'$.
Notice that we have $\mu_y \left( \bigcup B_i' \right) \geq 1 - \epsilon' - \epsilon/2 \geq 1 -\epsilon$, so we just need to estimate $\diam_{Q,F_n}(B_i')$.
If $x,z \in B_i'$, then
\begin{alignat*}{2}
d_{Q,F_n}(x,z) \ &= \ \frac{1}{|F_n|} \sum_{f \in F_n} 1_{Q(T^f x) \neq Q(T^f z)} && \\
&\leq \ \frac{1}{|F_n|} \sum_{f \in F_n} 1_{Q(T^f x) \neq Q'(T^f x)} && + \frac{1}{|F_n|} \sum_{f \in F_n} 1_{Q'(T^f x) \neq Q'(T^f z)} \\
& && + \frac{1}{|F_n|} \sum_{f \in F_n} 1_{Q'(T^f z) \neq Q(T^f z)} \\
&\leq \ \sqrt{\epsilon'} + \epsilon' + \sqrt{\epsilon'}. &&
\end{alignat*}
The estimates on the first and third terms hold because $x,z \in X'$, and the estimate on the second term holds because $\diam_{Q',F_n}(B_i') \leq \diam_{Q',F_n}(B_i) \leq \epsilon'$.
Therefore we have $\diam_{Q,F_n}(B_i') \leq 3\sqrt{\epsilon'} = 3\epsilon^2/4 \leq \epsilon$ as desired.
\end{proof}

\begin{lemma} \label{lemma: ApproximateByRefiningPartition}
Let $P$ be any partition of $X$, and let $Q'$ be another partition which is measurable with respect to $P^{F_k} \vee \pi^{-1}\mathcal{B}_Y$ for some $k \in \N$.
Then for any $\epsilon > 0$, there exists $\epsilon'$ such that
\[
\cov(\mu, Q', F_n, \epsilon \,|\, \pi) \ \leq \ \cov(\mu, P, F_n, \epsilon' \,|\, \pi) \qquad \text{for all $n$ sufficiently large.}
\]
\end{lemma}

\begin{proof}
Set $\epsilon' = \epsilon/(2|F_k|)$.
Let $\cov(\mu, P, F_n, \epsilon' \,|\, \pi) = L$ and let $S \subseteq Y$ be such that $\nu(S) \geq 1-\epsilon'$ and $\cov(\mu_y, P, F_n, \epsilon') \leq L$ for all $y \in S$.
We claim that also $\cov(\mu_y, Q', F_n, \epsilon) \leq L$ for $y \in S$.
Fix $y \in S$ and let $B_1, \dots, B_L$ be such that $\diam_{P,F_n}(B_i) \leq \epsilon'$ and $\mu_y \left( \bigcup B_i \right) \geq 1-\epsilon'$.
We may assume without loss of generality that each $B_i$ is contained in the fiber $\pi^{-1} y$ because doing so can only decrease their diameters and does not change their measure according to $\mu_y$.

We will be done as soon as we show that $\diam_{Q',F_n}(B_i) \leq \epsilon$.
To do that, let $x,z \in B_i$ and observe that because $x$ and $z$ lie in the same fiber of $\pi$, $Q'(x) \neq Q'(z)$ implies that $x$ and $z$ must be in different cells of $P^{F_k}$.
Therefore we can estimate
\begin{align*}
d_{Q',F_n}(x,z) \ &= \ \frac{1}{|F_n|} \sum_{f \in F_n} 1_{Q'(T^f x) \neq Q'(T^f z)} \ \leq \ \frac{1}{|F_n|} \sum_{f \in F_n} 1_{P^{F_k}(T^f x) \neq P^{F_k}(T^f z)} \\
&\leq \ \frac{1}{|F_n|} \sum_{f \in F_n} \sum_{g \in F_k} 1_{P(T^{gf} x) \neq P(T^{gf} z)} \\
&= \ \frac{1}{|F_n|} \sum_{h \in F_k F_n} \#\{ (g,f) \in F_k \times F_n : gf = h \} \cdot 1_{P(T^h x) \neq P(T^h z)} \\
&\leq \ \frac{1}{|F_n|} \sum_{h \in F_k F_n} |F_k| \cdot 1_{P(T^h x) \neq P(T^h z)} \\
&= \ \frac{1}{|F_n|} \sum_{h \in F_n} |F_k| \cdot 1_{P(T^h x) \neq P(T^h z)} + \frac{1}{|F_n|} \sum_{h \in F_k F_n \setminus F_n} |F_k| \cdot 1_{P(T^h x) \neq P(T^h z)} \\
&\leq \ |F_k| \cdot d_{P, F_n}(x,z) + \frac{|F_k| \cdot |F_k F_n \setminus F_n|}{|F_n|}
\end{align*}
So as soon as $n$ is sufficiently large, because $(F_n)$ is a F\o lner sequence, we have $\diam_{Q',F_n}(B_i) \leq 2|F_k|\epsilon' = \epsilon$ as desired.
\end{proof}

\begin{proof}[Proof of \Cref{theorem: GeneratingPartitionsDominate}]
To prove this theorem, it is clearly sufficient to prove the following finitary version: for any partition $Q$ of $X$ and any $\epsilon > 0$, there exist $m \in \N$ and $\epsilon' > 0$ such that
\[
\cov(\mu, Q, F_n, \epsilon \,|\, \pi) \ \leq \ \cov(\mu, P_m, F_n, \epsilon' \,|\, \pi)
\]
for all sufficiently large $n$.

Fix a partition $Q$ and $\epsilon > 0$.
Let $\epsilon'$ be as in the statement of \Cref{lemma: ApproximateByNearbyPartition}.
Then, by the definition of relatively generating sequence of partitions, we can find $m, k \in \N$ and a partition $Q'$ refined by $P_m^{F_k} \vee \pi^{-1}\mathcal{B}_Y$ such that $\dist_\mu(Q,Q') \leq \epsilon'$.
Apply \Cref{lemma: ApproximateByNearbyPartition} to conclude that
\beq \label{eq: GeneratingPartitionApprox1}
\cov(\mu, Q, F_n, \epsilon \,|\, \pi) \ \leq \ \cov(\mu, Q', F_n, \epsilon' \,|\, \pi)
\eeq
for $n$ sufficiently large.
Then apply \Cref{lemma: ApproximateByRefiningPartition} with $\epsilon'$ in place of $\epsilon$ to produce an $\epsilon''$ satisfying
\beq \label{eq: GeneratingPartitionApprox2}
\cov(\mu, Q', F_n, \epsilon' \,|\, \pi) \ \leq \ \cov(\mu, P_m, F_n, \epsilon'' \,|\, \pi)
\eeq
for sufficiently large $n$.
Combining \eqref{eq: GeneratingPartitionApprox1} and \eqref{eq: GeneratingPartitionApprox2} gives the desired result.
\end{proof}

\subsection{Relationship to Kolmogorov-Sinai entropy}

In this section, we show that with an exponential rate function, relative slow entropy recovers the classical relative Kolomogorov-Sinai entropy.
Given an extension 
$\pi: \X \to \Y$
and a partition $P$ of $X$, let $h_{\operatorname{KS}}(\X, P \,|\, \pi)$ denote the relative Kolmogorov-Sinai entropy rate (see e.g. \cite[Definition 2.18]{einsiedler2021entropy} for the definition).

\begin{definition}
Define
\[
h'(\X, P \,|\, \pi ) \ := \ \sup_{\epsilon > 0}\ \limsup_{n \to \infty} \frac{1}{|F_n|} \log \cov(\mu, P, F_n, \epsilon \,|\, \pi ).
\]
Note that \emph{a priori} this quantity depends on the choice of F\o lner sequence, but the next theorem shows that it actually does not. 
\end{definition}

\begin{theorem} \label{theorem: RecoverKSEntropy}
Assume that $\X$ is ergodic.
For any partition $P$ of $X$, we have $h'(\X, P \,|\, \pi) = h_{\operatorname{KS}}(\X, P \,|\, \pi)$.
\end{theorem}

The non-relative version of this result appears in \cite[Proposition 2]{ferenczi1997complexity}

\begin{proof}
For this proof, we will abuse notation and write $P^{F_n}(x)$ to mean the cell of the partition $P^{F_n}$ that contains $x$, rather than the name of the cell.
Abbreviate $h = h_{\operatorname{KS}}(\X, P \,|\, \pi)$ and $h' = h'(\X, P \,|\, \pi)$.

\textbf{Step 1: setup.}
Fix an arbitrary $\gamma > 0$.
For $y \in Y$, define
\[
\mathcal{G}_{y,n} \ := \ \left\{ x \in X : \exp( |F_n| (-h - \gamma) ) < \mu_{y} \left( P^{F_n}(x) \right) < \exp( |F_n| (-h + \gamma) ) \right\}.
\]
Also define
\[
\mathcal{G}_n \ := \ \left\{ x \in X : \exp( |F_n| (-h - \gamma) ) < \mu_{\pi x} \left( P^{F_n}(x) \right) < \exp( |F_n| (-h + \gamma) ) \right\}.
\]
By the amenable version of the relative Shannon-McMillan theorem (see for example \cite[Theorem 3.2]{ward1992amenable} or \cite[Corollary 4.6]{rudolph2000mixing}), $\mu(\mathcal{G}_n) \to 1$ as $n \to \infty$.
Because $\mu = \int \mu_y \,d\nu(y)$, this implies also $\mu_y(\mathcal{G}_n) \to 1$ for $\nu$-a.e. $y$.
Finally, since $\mu_y$ is supported only on $\pi^{-1} y$, we have $\mu_y(\mathcal{G}_{y,n}) = \mu_y(\mathcal{G}_n) \to 1$ for $\nu$-a.e. $y$.

\textbf{Step 2: $\mathbf{h' \leq h}$.}
Let $\epsilon > 0$.
Because $\mu_y(\mathcal{G}_{y,n}) \to 1$ for $\nu$-a.e. $y$, for all $n$ sufficiently large we can find a set $S \subseteq Y$ such that $\nu(S) \geq 1-\epsilon$ and $\mu_y(\mathcal{G}_{y,n}) \geq 1-\epsilon$ for all $y \in S$.
We now estimate $\cov(\mu, P, F_n, \epsilon \,|\, \pi)$.
Because $\nu(S) \geq 1-\epsilon$, it suffices to estimate $\cov(\mu_y, P, F_n, \epsilon)$ for $y \in S$.
Fix such a $y$; we claim that $\cov(\mu_y, P, F_n, \epsilon) \leq \exp(|F_n| (h+\gamma))$.
Let $C_1, \dots, C_L$ be all of the cells of $P^{F_n}$ that meet $\mathcal{G}_{y,n}$.
Then each $\diam_{P,F_n}(C_i) = 0$ and
\[
\mu_y \left( \bigcup C_i \right) \ = \ \mu_y (\mathcal{G}_{y,n}) \ \geq \ 1-\epsilon,
\]
so $\cov(\mu_y, P, F_n, \epsilon) \leq L$.
But by definition of $\mathcal{G}_{y,n}$, each $C_i$ has $\mu_y-$measure at least 
\[
\exp(|F_n| (-h-\gamma)),
\] 
so $L \leq \exp(|F_n| (h+\gamma))$ as claimed.
Since this holds for all $y \in S$, this shows that 
\[
\cov(\mu, P, F_n, \epsilon \,|\, \pi) \ \leq \ \exp(|F_n| (h+\gamma))
\]
for sufficiently large $n$.
Therefore we can take $n \to \infty$ and then $\epsilon \to 0$ to conclude $h' \leq h+\gamma$.
But since $\gamma$ is arbitrary we get $h' \leq h$ as desired.

\textbf{Step 3: $\mathbf{h' \geq h}$.}
Again fix $0 < \epsilon < 1/4$ and let $n$ and
 $S$ be as in step $2$.
Also let us enumerate $P = \{P_0, \dots, P_{r-1} \}$.
This time we will estimate a lower bound for $\cov(\mu_y, P, F_n, \epsilon)$ for any $y \in S$.
Suppose $E_1, \dots, E_M$ are sets in $X$ satisfying $\diam_{P,F_n}(E_i) \leq \epsilon$ and $\mu_y( \bigcup E_i ) \geq 1-\epsilon$.
Without loss of generality, we may assume that each $E_i$ is a union of $P^{F_n}$-cells because replacing each $E_i$ by $\bigcup_{x \in E_i} P^{F_n}(x)$ makes each $E_i$ larger without changing its diameter according to $d_{P,F_n}$.
Because $y \in S$, we then have
\[
\mu_y\left( \bigcup E_i \cap \mathcal{G}_{y,n} \right) \ \geq \ 1-2\epsilon,
\]
and we can estimate
\[
1/2 \ \leq \ 1-2\epsilon \ \leq \ \mu_y\left( \bigcup E_i \cap \mathcal{G}_{y,n} \right) \ \leq \ \sum_{i=1}^{M} \mu_y (E_i \cap \mathcal{G}_{y,n}).
\]
Each $E_i$ is a union of $P^{F_n}$-cells, and any $P^{F_n}$-cell meeting $\mathcal{G}_{y,n}$ has $\mu_y$-measure at most 
\[
    \exp(|F_n| (-h+\gamma))
\]
by construction.
Therefore the above sum is at most
\[
\sum_{i=1}^{M} \exp(|F_n| (-h+\gamma)) \cdot (\# \text{ of $P^{F_n}$-cells contained in $E_i$}).
\]
Since $E_i$ has diameter $\leq \epsilon$ according to $d_{P,F_n}$, the elements of $\{0,1,\dots, r-1\}^{F_n}$ corresponding to the $P^{F_n}$-cells contained in $E_i$ all fit inside a fixed ball of radius $\epsilon$ in the normalized Hamming metric on $\{0,1,\dots, r-1\}^{F_n}$.
It is well known (for example, it is an easy consequence of \cite[Lemma 3.6]{gray2011entropy}) that the number of words in any such Hamming ball is at most 
\[
    \exp\Big(|F_n| \cdot (\epsilon \log(r-1) + H(\epsilon, 1-\epsilon))\Big),
\]
where $H$ is the Shannon entropy function $H(t,1-t) = -t \log t - (1-t) \log (1-t)$.
Therefore the above sum is bounded by
\[
M \cdot \exp\Big(|F_n| (-h + \gamma + \epsilon \log(r-1) + H(\epsilon, 1-\epsilon))\Big),
\]
implying that 
\[
    M \ \geq \ (1/2) \exp\Big(|F_n| (h - \gamma - \epsilon\log(r-1) + H(\epsilon, 1-\epsilon))\Big) \ =: \ Z.
\]
Since we started with an arbitrary covering set this implies that $\cov(\mu_y, P, F_n, \epsilon) \geq  Z$ for every $y \in S$.
Since $\nu(S) > 1-\epsilon$, it has positive $\nu$-measure intersection with any other set of $\nu$-measure $\geq 1-\epsilon$, so it is impossible to find a different set $S'$ with $\mu(S') > 1-\epsilon$ and $\cov(\mu_y, P, F_n, \epsilon) \leq  Z$ for every $y \in S'$.
Therefore $\cov(\mu, P, F_n, \epsilon \,|\, \pi) \geq  Z$.
Taking $n \to \infty$ gives
\[
\limsup_{n \to \infty} \frac{1}{|F_n|} \log \cov(\mu, P, F_n, \epsilon \,|\, \pi) \ \geq \ h - \gamma - \log(r) \cdot H(\epsilon, 1-\epsilon),
\]
then taking $\epsilon \to 0$ gives $h' \geq h-\gamma$, and again $\gamma$ is arbitrary so we conclude $h' \geq h$.
\end{proof}

\section{Isometric extensions} \label{sec: Isometric}

Let $\X$ be ergodic and let $\pi: \X \to \Y$ be a factor map.
For this section, a {\bf cocycle} from $Y$ to a compact group $H$ is defined to be a measurable map $\a: G \times Y \to H$ satisfying the cocycle condition $\a(g'g, y) = \a(g', S^g y) \a(g, y)$ for every $g,g' \in G$ and $\nu$-a.e. $y$.

Given a cocycle $\a$ from $Y$ to $H$ and a closed subgroup $K \subseteq H$, we denote by $\mathbf{Y \times_\a H/K}$ the system $(Y \times H/K, \nu \times m_{H/K}, T_\a)$, where
\begin{itemize}
    \item $m_{H/K}$ is the image of the Haar measure $m_H$ under the quotient map $H \to H/K$, and
    \item $T_\a^g(y, hK) := (S^g y, \a(g,y)hK)$. 
\end{itemize}
Note that $\mathbf{Y \times_\a H/K}$ is an extension of $\Y$ via the projection map onto the first coordinate.
Such an extension is also called a \textbf{homogeneous skew product} over $\Y$.

\begin{definition} \label{definition: IsometricExtension}
An extension 
$\pi: \X \to \Y$
is said to be \textbf{isometric} if 
$\pi$
is isomorphic to the projection map $\mathbf{Y \times_\a H/K} \to \Y$, i.e. there is a commutative diagram
\begin{center}
\begin{tikzcd}
    \X \arrow[r, "\sim"] \arrow[d, "\pi"] & \mathbf{Y \times_\a H/K} \arrow[d, "\operatorname{proj}"] \\
    \Y \arrow[r, "\operatorname{id}"] & \Y 
\end{tikzcd}.   
\end{center}
\end{definition}

\begin{remark}
    There are several different equivalent definitions in the literature for what it means for $\pi$ to be an isometric extension, but the above is the most convenient choice for our purposes.
    Also, the terms ``isometric extension'' and ``compact extension'' are often used interchangeably.
    The two notions are formally different, but they are known to be equivalent.
    See \cite[Definition 4, Definition 15, Theorem 22]{zorinkranichCompactIsometric} and \cite[Definition 9.10, Theorem 9.14]{glasner2003ergodic} for more details.
\end{remark}

\begin{definition}
We say $\pi: \X \to \Y$ has {\bf bounded complexity} with respect to the F\o lner sequence $(F_n)$ if 
$\h^{U, (F_n)}(\X \,|\, \pi) = 0$
for every rate function $U$.
Equivalently, $\pi$ has bounded complexity if for every $\epsilon > 0$ and every partition $P$ of $X$, $\limsup_{n \to \infty} \cov(\mu, P, F_n, \epsilon \,|\, \pi) < \infty$.
\end{definition}

The purpose of this section is to prove the following characterization of isometric extensions.

\begin{theorem} \label{theorem: IsometricEquivBounded}
    Suppose $\X$ is ergodic and let
    $\pi : \X \to \Y$
    be an extension.  Then the following are equivalent.
    \begin{enumerate}
        \item $\pi$ is isometric.
        \item $\pi$ has bounded complexity with respect to some F\o lner sequence.
        \item $\pi$ has bounded complexity with respect to every F\o lner sequence.
    \end{enumerate}
\end{theorem}

\subsection{Isometric implies bounded}

\begin{proposition} \label{proposition: IsometricImpliesBounded}
If $\pi$ is isometric, then it has bounded complexity with respect to any F\o lner sequence.
\end{proposition}

Every homogeneous skew product as in \Cref{definition: IsometricExtension} is a factor of a group rotation skew product on $Y \times H$ (i.e. a homogeneous skew product with $K$ the trivial subgroup).
So, by \Cref{theorem: MonotoneUnderExtensions}, it is sufficient to assume that $X = Y \times H$, $\mu = \nu \times m_H$, and $T^g(y,h) = (S^gy, \a(g, y)h)$ for some cocycle $\a$ from $Y$ to $H$.
Let $\rho$ be a translation-invariant metric on $H$, and let $(F_n)$ be any choice of F\o lner sequence for $G$.

\begin{proposition} \label{proposition: IsometricImpliesBoundedForSpecialPartitions}
Let $Q = \{Q_1, \dots, Q_k\}$ be a partition of $H$.
Let $P$ be the partition $\{Y \times Q_1, \dots, Y \times Q_k\}$ of $X$.
Then for any $\epsilon > 0$, there exists $L = L(Q, \epsilon)$ such that
\[
\cov(\mu, P, F_n, \epsilon \,|\, \pi) \ \leq \ L
\]
for all $n$ sufficiently large.
\end{proposition}

\begin{proof}
For each $i$, let $Q_i'$ be a compact subset of $Q_i$ such that 
\[
    m_H(Q_i') \ > \ (1-\epsilon/4) \cdot m_H(Q_i).
\]
Let $E = H \setminus \bigcup Q_i'$, so $m_H(E) \leq \epsilon/4$.
Let $\bar{E} = Y \times E$, so $\mu(\bar{E}) \leq \epsilon/4$.
Now because the $Q_i'$ are pairwise disjoint compact sets, there is some $\delta = \delta(Q, \epsilon)$ such that $\rho(Q_i', Q_j') \geq \delta$ for all $i\neq j$.
Let $L$ be the smallest number of balls of $\rho$-radius at most $\delta/2$ required to cover $H$, and let $B_1, \dots, B_L \subseteq H$ be a collection of such balls.
Let $\bar{B_i} = Y \times B_i$.

We claim that $\cov(\mu, P, F_n, \epsilon \,|\, \pi) \leq L$ for all $n$ sufficiently large.
By the mean ergodic theorem and the fact that $\mu(\bar{E}) \leq \epsilon/4$, we have
\[
\mu \left\{ (y,h) : \frac{1}{|F_n|} \sum_{g \in F_n} 1_{\bar{E}}(T^g(y,h)) < \epsilon/2 \right\} \ \to \ 1
\]
as $n \to \infty$.
Call this set $X'$, and let $n$ be sufficiently large so that $\mu(X') \geq 1-\epsilon^2$.
By Markov's inequality, we have a set $S \subseteq Y$ with $\nu(S) \geq 1-\epsilon$ such that $\mu_y(X') \geq 1-\epsilon$ for all $y \in S$.
Fix $y \in S$; we now estimate $\cov(\mu_y, P, F_n, \epsilon)$.
Let $B_i' = \bar{B_i} \cap X' \cap \pi^{-1}y$.
Because the $\bar{B_i}$ cover all of $X$, we have
\[
\mu_y \left( \bigcup B_i' \right) \ = \ \mu_y(X') \ \geq \ 1-\epsilon.
\]
Therefore we just need to estimate $\diam_{P,F_n}(B_i')$.

Suppose $(y,h), (y,h') \in B_i'$.
We have
\begin{align*}
d_{P,F_n}((y,h), (y,h')) \ &= \ \frac{1}{|F_n|} \sum_{g \in F_n} 1_{P(T^g(y,h)) \neq P(T^g(y,h'))} \\
&= \ \frac{1}{|F_n|} \sum_{g \in F_n} 1_{P(S^gy, \a(g, y)h) \neq P(S^g y, \a(g, y)h')} \\
&= \ \frac{1}{|F_n|} \sum_{g \in F_n} 1_{Q(\a(g, y)h) \neq Q(\a(g, y)h')}.
\end{align*}
By definition of the $B_i$, we have $\rho(h,h') \leq \delta/2$, and because $\rho$ is translation invariant, we also have $\rho(\a(g, y)h, \a(g, y)h') \leq \delta/2$ for all $g$.
Therefore if $\a(g, y)h$ and $\a(g, y)h'$ are in different cells of $Q$, it must be the case that either $\a(g, y)h \in E$ or $\a(g, y)h' \in E$.
So the above becomes
\begin{align*}
d_{P,F_n}((y,h), (y,h')) \ &\leq \ \frac{1}{|F_n|} \sum_{g \in F_n} 1_{E}(\a(g, y)h) + 1_{E}(\a(g, y)h') \\
&= \ \frac{1}{|F_n|} \sum_{g \in F_n} 1_{\bar{E}}(T^g(y,h)) + 1_{\bar{E}}(T^g(y,h')) \\
&\leq \ \epsilon
\end{align*}
because $(y,h), (y,h') \in X'$.
So we have $\diam_{P,F_n}(B_i') \leq \epsilon$.
This shows that 
\[
\cov(\mu_y, P, F_n, \epsilon) \ \leq \ L
\]
for all $y \in S$ and $n$ sufficiently large as desired.
\end{proof}

\begin{proof}[Proof of \Cref{proposition: IsometricImpliesBounded}]
For each $m$, let $Q_m$ be a partition of $H$ into sets of diameter at most $1/m$ and let $P_m = \{ Y \times C : C \in Q_m \}$ as in \Cref{proposition: IsometricImpliesBoundedForSpecialPartitions}.
It is clear that the sequence $(P_m)_{m=1}^{\infty}$ is generating for $\X$ relative to $\pi$.
By \Cref{proposition: IsometricImpliesBoundedForSpecialPartitions}, we have $\limsup_{n \to \infty} \cov(\mu, P_m, F_n, \epsilon \,|\, \pi) < \infty$ for every $m$ and every $\epsilon > 0$.
Then by \Cref{theorem: GeneratingPartitionsDominate}, for any partition $R$ of $X$, we have
\[
\h^{U, (F_n)}(\X, R \,|\, \pi) \ \leq \ \h^{U, (F_n)}(\X \,|\, \pi) \ = \ \lim_{m \to \infty} \h^{U, (F_n)}(\X, P_m \,|\, \pi) \ = \ 0
\]
for any rate function $U$, as desired.
\end{proof}

\subsection{Background on conditional weak mixing}

The second half of the proof of \Cref{theorem: IsometricEquivBounded} requires the theory of compact and weakly mixing extensions originally developed in \cite{furstenberg1977szemeredi}.
All of the necessary background material presented here can be found in \cite[Chapter 3 and Appendix D]{KerrLi2016Ergodic}.

\begin{definition} \label{definition: AbsoluteDensity}
    We say that a subset $\Gamma \subseteq G$ has \textbf{absolute density $1$} if
    \[
    \lim_{n \to \infty} \frac{|\Gamma \cap F_n|}{|F_n|} \ = \ 1
    \]
    for \emph{any} F\o lner sequence $(F_n)$.
\end{definition}

\begin{definition}
For $y \in Y$ and $f,g \in L^2(\X)$, define
\[
\ab{f,g}_y \ := \ \int f \bar{g} \,d\mu_y.
\]
Let $L^2(\X \,|\,\pi)$ denote the space of $f \in L^2(\X)$ such that $y \mapsto \ab{f,f}_y \in L^\infty(\Y)$.
We also say that $f,g \in L^2(\X \,|\, \pi)$ are \textbf{conditionally orthogonal given $\pi$} if $\ab{f,g}_y = 0$ for $\nu$-a.e. $y \in Y$.
\end{definition}

In this section, we identify the action $T$ with its Koopman representation on $L^2(\X)$.
So, for $g \in G$ and $f \in L^2(\X)$, we write $T^g f$ to mean $f \circ T^g$.

\begin{definition} \label{definition: ConditionalWeakMixing}
A function $f \in L^2(\X \,|\, \pi)$ is said to be \textbf{conditionally weakly mixing given $\pi$} if 
for any F\o lner sequence $(F_n)$ and any $g \in L^2(\X\,|\,\pi)$, we have
\[
    \lim_{n \to \infty} \frac{1}{|F_n|} \sum_{s \in F_n} \int \abs{ \ab{T^s f, g}_y } \,d\nu(y) \ = \ 0.
\]
Equivalently, for any $g \in L^2(\X\,|\,\pi)$ and any $\epsilon > 0$, the set
\[
    \Gamma_{f,g,\epsilon} \ := \ \left\{ s \in G : \int \abs{ \ab{T^s f, g}_y } \,d\nu(y) < \epsilon \right\}
\]
has absolute density $1$.
The set of all conditionally weakly mixing functions is denoted $W(\X \,|\, \pi)$.
\end{definition}

The main fact we will need to use is the following characterization of the maximal intermediate isometric extension (essentially \cite[Proposition 3.9 and Lemma 3.11]{KerrLi2016Ergodic}).

\begin{theorem}
\label{theorem: RelativeKroneckerProperties}
Let $\X$ be ergodic and 
$\pi: \X \to \Y$
be an extension.
Then there exists an intermediate
extension $\X \to \mathbf{Z} \to \Y$
such that
\begin{itemize}
    \item 
    $\mathbf{Z}$
    is the maximal isometric extension of $\Y$ in $\X$, and
    \item For $f \in L^2(\X \,|\, \pi)$, $f \in W(\X \,|\, \pi)$ if and only if $f$ is conditionally orthogonal to every 
    $\mathbf{Z}$-measurable $h \in L^2(\X \,|\, \pi)$.
\end{itemize}
\end{theorem}

\subsection{Bounded implies isometric}

\begin{proposition} \label{proposition: BoundedImpliesIsometric}
If $\pi$ has bounded complexity with respect to some F\o lner sequence $(F_n)$, then it is isometric.
\end{proposition}

Suppose for contradiction that $\pi$ is not isometric but does have bounded complexity with respect to some F\o lner sequence $(F_n)$.
Let 
$\mathbf{Z}$ be the 
maximal intermediate isometric extension as in \Cref{theorem: RelativeKroneckerProperties}.
Because of the assumption that $\pi$ is not isometric, we know that 
$\mathbf{Z}$ is a strict factor of $\X$,
so we can choose a partition $P = \{P_0, P_1\}$ of $X$ satisfying
\begin{enumerate}
    \item $P$ is independent of 
    $\mathbf{Z}$ 
    \item $\mu_y(P_0), \mu_y(P_1) \geq 1/3$ for $\nu$-a.e. $y$.
    \footnote{If the factor map $\X \to \mathbf{Z}$ is infinite-to-one, then $1/3$ may be replaced by $1/2$.}
\end{enumerate}
Fix this partition for the rest of this section.
Also let $0 < \epsilon < 10^{-6}$ be fixed.
Finally, using the notation of \Cref{definition: PartitionDistance}, for $y \in Y$ we abbreviate $\dist_y := \dist_{\mu_y}$.

The outline of the proof of \Cref{proposition: BoundedImpliesIsometric} is as follows.
First, using the assumption that $\pi$ has bounded complexity, we will find a ``positive density'' set of pairs of times $(s,t) \in G^2$ such that $T^{s^{-1}}P$ and $T^{t^{-1}}P$ are very close to each other (in most of the fibers of $\pi$).  
This is \Cref{lemma: CloseTogetherPartition}.
Then, using the independence conditions built in to the definition of $P$, we show essentially that the partition $P$ is conditionally weakly mixing given $\pi$, which allows us to find a ``density one'' set of pairs of times $(s,t) \in G^2$ such that $T^{s^{-1}}P$ and $T^{t^{-1}}P$ are approximately independent of each other (in most of the fibers of $\pi$).
This is \Cref{lemma: IndependentPartition}.
Therefore we can find a pair of times $(s,t)$ for which $T^{s^{-1}}P$ and $T^{t^{-1}}P$ are both close together and approximately independent of each other.
But it is impossible for two nontrivial partitions to satisfy this, so we will get a contradiction.

\begin{lemma} \label{lemma: CloseTogetherPartition}
For $y \in Y$, define
\[
\C_y \ := \ \left\{(s,t) \in G^2 : \dist_y(T^{s^{-1}}P, T^{t^{-1}}P) \ < 5\sqrt{\epsilon} \right\}.
\]
Then there is a constant $c = c(\epsilon) > 0$ such that the following holds.
For every $n$, there is a set $Y_n \subseteq Y$ satisfying $\nu(Y_n) \geq 1-\epsilon$ and
\[
    \frac{|\C_y \cap F_n^2|}{|F_n^2|} \ \geq \ c(\epsilon)
\]
for all $y \in Y_n$.
\end{lemma}

\begin{proof}
Let $L = L(\epsilon) = \sup_{n} \cov(\mu, P, F_n, \epsilon \,|\, \pi)$ and let $n$ be arbitrary.
Let $Y_n$ be the set of $y \in Y$ such that $\cov(\mu_y, P, F_n, \epsilon) \leq L$.
We have $\nu(Y_n) \geq 1 - \epsilon$ by definition.
For the rest of this proof, let $y \in Y_n$ be fixed.
We seek to bound $|\C_y \cap F_n^2|/|F_n^2|$ from below by a quantity depending only on $\epsilon$.

Let $B_1, \dots, B_L$ be subsets of $X$ such that each $B_i$ has $d_{P,F_n}$-diameter at most $\epsilon$ and $\mu_y\left( \bigcup B_i \right) \geq 1 - \epsilon$.
Let $X' = \bigcup B_i$.
Without loss of generality, we may assume that the $B_i$ are disjoint.
For each $i$, fix a point $x_i \in B_i$.
Then, for $x \in X'$, define $r(x)$ to be the unique $x_i$ such that $x \in B_i$.

By construction, we know that for each $x \in X'$, $T^s x$ and $T^s r(x)$ lie in the same $P$-cell for most $s \in F_n$, but the set of good ``times'' $s$ changes as $x$ varies.
We now apply a form of Markov's inequality to upgrade this to the statement that for most $s \in F_n$, $\mu_y$-most $x$ satisfy $P(T^s x) = P(T^s r(x))$. 
Define
\[
A \ = \ \Big\{ s \in F_n : \mu_y \{x \in X' : P(T^s x) = P(T^s r(x)) \} \geq 1-\sqrt{2\epsilon} \Big\}.
\]
We have
\begin{align*}
    \sum_{s \in F_n} \mu_y\{x \in X' : P(T^s x) = P(T^s r(x)) \} \ &= \ \sum_{s \in F_n} \sum_{i=1}^{L} \mu_y\{x \in B_i : P(T^s x) = P(T^s x_i) \} \\
    &= \ \sum_{s \in F_n} \sum_{i=1}^{L} \int_{B_i} 1_{P(T^s x) = P(T^s x_i)} \,d\mu_y(x) \\
    &= \ \sum_{i=1}^{L} \int_{B_i} \sum_{s \in F_n} 1_{P(T^s x) = P(T^s x_i)} \,d\mu_y(x) \\
    &= \ \sum_{i=1}^{L} \int_{B_i} |F_n| (1 - d_{P,F_n}(x,x_i)) \,d\mu_y(x) \\
    &\geq \ \sum_{i=1}^{L} \mu_y(B_i) |F_n| (1-\epsilon) \\
    &\geq \ |F_n| (1-\epsilon)^2 \ \geq \ |F_n| (1-2\epsilon).
\end{align*}
But the original sum above also satisfies
\begin{align*}
\sum_{s \in F_n} \mu_y\{x \in X' : P(T^s x) = P(T^s r(x)) \} \ &= \ \sum_{s \in A} \mu_y\{x \in X' : P(T^s x) = P(T^s r(x)) \} \ + \\
& \quad \ \, \sum_{s \not\in A} \mu_y\{x \in X' : P(T^s x) = P(T^s r(x)) \} \\
&\leq \ |A| + (|F_n|-|A|)(1-\sqrt{2\epsilon}) \\
&= \ |F_n|(1-\sqrt{2\epsilon}) + |A|\cdot \sqrt{2\epsilon}.
\end{align*}
Combining these two inequalities shows that
\beq \label{eq: MostPointsAreCloseMostOfTheTime}
|A| \ \geq \ \frac{|F_n|(1-2\epsilon - (1-\sqrt{2\epsilon}))}{\sqrt{2\epsilon}} \ = \ |F_n|(1-\sqrt{2\epsilon}).
\eeq

The set $A$ decomposes as
\[
A \ = \ \bigcup_{w \in  \{0,1\}^L} \{s \in A : (P(T^s x_i))_{i=1}^L = w \}.
\]
By \eqref{eq: MostPointsAreCloseMostOfTheTime} and the pigeonhole principle, there is some $w \in \{0,1\}^L$ such that
\beq \label{equation: SizeOfEy}
\abs{ \{s \in A : (P(T^s x_i))_{i=1}^L = w \} } \ \geq \ |F_n|(1-\sqrt{2\epsilon}) \cdot 2^{-L}.
\eeq
Call this set $\mathcal{E}$. 
For $s,t \in \mathcal{E}$, 
say that $x$ is \emph{$(s,t)$-good} if $P(T^s x) = P(T^s r(x))$ and $P(T^t x) = P(T^t r(x))$.
By definition of $A$, the set of $x$ that are not $(s,t)$-good has $\mu_y$-measure at most $2\sqrt{2\epsilon}$.
Now, for 
$s,t \in \mathcal{E}$,
we can estimate
\begin{align*}
    \dist_y(T^{s^{-1}}P, T^{t^{-1}}P) \ &= \ \int 1_{P(T^s x) \neq P(T^t x)} \,d\mu_y(x) \ \leq \ \epsilon + \int_{X'} 1_{P(T^s x) \neq P(T^t x)} \,d\mu_y(x) \\
    &\leq \ \epsilon + 2\sqrt{2\epsilon} + \int_{ \{x \in X' :\ x \text{ is $(s,t)$-good}\} } 1_{P(T^s x) \neq P(T^t x)} \,d\mu_y(x) \\
    &= \ \epsilon + 2\sqrt{2\epsilon} + \sum_{i=1}^{L} \mu_y(B_i) 1_{P(T^s x_i) \neq P(T^t x_i)}.
\end{align*}
By definition of 
$\mathcal{E}$,
$P(T^s x_i) = w_i = P(T^t x_i)$ for all $i$, so this final sum vanishes and we conclude $\dist_y(T^{s^{-1}}P, T^{t^{-1}}P) \leq \epsilon + 2\sqrt{2\epsilon} \leq 5\sqrt{\epsilon}$ whenever $s,t \in \mathcal{E}$.

Finally, observe that $\C_y$ contains $\mathcal{E} \times \mathcal{E}$.
Therefore, by \eqref{equation: SizeOfEy}, we have
\[
    \frac{|\C_y \cap F_n^2|}{|F_n^2|} \ \geq \ \left((1-2\sqrt{\epsilon}) \cdot 2^{-L(\epsilon)} \right)^2 \ > \ 0    
\]
as claimed.
\end{proof}

This finishes the first half of our outline.
For convenience, we now introduce some new definitions before starting the second half.

\begin{definition}
Given $y \in Y$ and two sets $A, B \susbeteq X$, we define the \textbf{dependence score} with respect to $\mu_y$ to be
\[
\mathscr{D}_y(A,B) \ := \ \abs{ \mu_y(A \cap B) - \mu_y(A)\mu_y(B)}.
\]
We also define the \textbf{averaged dependence score}
\[
    \mathscr{D}(A,B) \ = \ \int \mathscr{D}_y(A,B) \,d\nu(y).
\]
Finally, if $Q$ and $Q'$ are two finite partitions of $X$, then the averaged dependence score between $Q$ and $Q'$ is defined to be
\[
    \mathscr{D}(Q,Q') \ = \ \max_{i,j} \mathscr{D}(Q_i, Q'_j).
\]
\end{definition}

\begin{lemma} \label{lemma: TranslatedFolnerSequence}
    Let $(F_n)$ be a F\o lner sequence for $G$ and let $(g_n)$ be an arbitrary sequence of elements of $G$.  Then $(F_n g_n)$ is also a F\o lner sequence for $G$.
\end{lemma}
\begin{proof}
    For any $h \in G$, we have 
    \[
        \frac{|h F_n g_n \cap F_n g_n|}{|F_n g_n|} \ = \ \frac{|(hF_n \cap F_n) g_n|}{|F_n g_n|} \ = \ \frac{|hF_n \cap F_n|}{|F_n|} \ \to \ 1  
    \]
    as $n \to \infty$.
\end{proof}

\begin{lemma} \label{lemma: DiagonalSetDensity}
    Let $\Gamma \subseteq G$ be a subset of absolute density $1$.
    Define $\Gamma' = \{(s,t) \in G^2 : ts^{-1} \in \Gamma \}$.
    Then if $(F_n)$ is any F\o lner sequence for $G$, we have 
    \[
    \lim_{n \to \infty} \frac{|\Gamma' \cap F_n^2|}{|F_n^2|} \ = \ 1.
    \]
\end{lemma}
\begin{proof}
Let $(F_n)$ be a left F\o lner sequence for $G$.
We calculate
\begin{align*}
    |\Gamma' \cap F_n^2| \ &= \ \#\{(s,t) \in F_n^2 : ts^{-1} \in \Gamma \} \\
    &= \ \sum_{s \in F_n} \#\{ t \in F_n : ts^{-1} \in \Gamma \} \\
    &= \ \sum_{s \in F_n} |F_n \cap \Gamma s| \\
    &= \ \sum_{s \in F_n} |F_n s^{-1} \cap \Gamma | \\
    &\geq \ |F_n| \cdot |F_n s_n^{-1} \cap \Gamma |,
\end{align*}
where $s_n \in F_n$ is defined to be the element of $F_n$ that minimizes $|F_n s^{-1} \cap \Gamma|$ over all $s \in F_n$.
By \Cref{lemma: TranslatedFolnerSequence}, $(F_n s_n^{-1})$ is also a F\o lner sequence, so because $\Gamma$ has absolute density $1$ we get
\begin{align*}
\lim_{n \to \infty} \frac{|\Gamma' \cap F_n^2|}{|F_n^2|} \ &\geq \ \lim_{n \to \infty} \frac{|F_n| \cdot |F_n s_n^{-1} \cap \Gamma |}{|F_n|^2} \\
&= \ \lim_{n \to \infty} \frac{|F_n s_n^{-1} \cap \Gamma |}{|F_n|} \ = \ \lim_{n \to \infty} \frac{|F_n s_n^{-1} \cap \Gamma |}{|F_n s_n^{-1}|} \ = \ 1.
\end{align*}
\end{proof}

Now fix another parameter $0 < \eta \ll \epsilon$ which is small enough so that $3\eta^{1/4} < c(\epsilon)/2$, where $c(\epsilon)$ is the quantity from \Cref{lemma: CloseTogetherPartition}.

\begin{lemma} \label{lemma: IndependentPartition}
For $y \in Y$, define
\[
\mathcal{I}_y \ := \ \left\{ (s,t) \in G^2 : \mathscr{D}_y(T^{s^{-1}}P, T^{t^{-1}}P) \leq \sqrt{\eta} \right\}.
\]
Then, for all sufficiently large $n$, there is a set $Y_n^\dagger \subseteq Y$ such that $\nu(Y_n^\dagger) \geq 1 - 3\eta^{1/4}$ and 
\[
\frac{|\mathcal{I}_y \cap F_n^2|}{|F_n^2|} \ \geq \ 1-3\eta^{1/4}
\]
for all $y \in Y_n^\dagger$.
\end{lemma}

\begin{proof}
Property (1) in the definition of the partition $P$ implies that if $f$ is any $P$-measurable function satisfying $\int f \,d\mu_y = 0$ for $\nu$-a.e. $y$, and $h$ is any 
$\mathbf{Z}$-measurable 
function, then also $\ab{f,h}_y = 0$ for $\nu$-a.e. $y$.
By the second bullet point of \Cref{theorem: RelativeKroneckerProperties}, this implies that any such $f$ is conditionally weak mixing given 
$\pi$.

Let $f_0 = 1_{P_0} - \mu(P_0)$ and $f_1 = 1_{P_1} - \mu(P_1)$.
Clearly these are both $P$-measurable, and because $P$ is independent of $\mathbf{Z}$ and therefore also independent of
$\Y$,
we also have $\int f_0 \,d\mu_y = \int f_1 \,d\mu_y = 0$ for $\nu$-a.e. $y$.
Therefore $f_0$ and $f_1$ are both conditionally weakly mixing.
Observe that
\begin{align*}
    \abs{ \ab{f_0,T^s f_1}_y } \ &= \ \abs{ \int f_0 \cdot T^s f_1 \,d\mu_y } \ = \ \abs{ \int (1_{P_0} - \mu(P_0))(1_{T^{s^{-1}}P_1} - \mu(P_1)) \,d\mu_y } \\
    &= \ \abs{ \mu_y(P_0 \cap T^{s^{-1}}P_1) - \mu(P_0)\mu_y(T^{s^{-1}}P_1) - \mu(P_1)\mu_y(P_0) + \mu(P_0)\mu(P_1) } \\
    &= \ \abs{ \mu_y(P_0 \cap T^{s^{-1}}P_1) - \mu_y(P_0)\mu_y(T^{s^{-1}}P_1) } \\
    &= \ \mathscr{D}_y(P_0, T^{s^{-1}}P_1).
\end{align*}
In the second to last line we again used the fact that $P$ is independent of 
$\Y$.
So by the discussion in \Cref{definition: ConditionalWeakMixing}, the set 
\[
\Gamma_{0,1} \ := \ \left\{ s \in G : \mathscr{D}(P_0, T^{s^{-1}}P_1) < \eta \right\}
\]
has absolute density $1$.

Applying the same analysis to $\ab{f_0,T^s f_0}_y$, $\ab{f_1,T^s f_0}_y$, and $\ab{f_1,T^s f_1}_y$ gives the same conclusion for each of the sets
\[
\Gamma_{i,j} \ := \ \left\{ s \in G : \mathscr{D}(P_i, T^{s^{-1}}P_j) < \eta \right\}.
\]
It follows that the set
\[
\Gamma \ := \ \bigcap_{0 \leq i,j \leq 1} \Gamma_{i,j} \ = \ \left\{ s \in G : \mathscr{D}(P, T^{s^{-1}}P) < \eta \right\}
\]
also has absolute density $1$.

As in \Cref{lemma: DiagonalSetDensity}, we now define the set of pairs
\[
\Gamma' \ := \ \left\{ (s,t) \in G^2 : \mathscr{D}(T^{s^{-1}}P, T^{t^{-1}}P) < \eta \right\} \ = \ \left\{ (s,t) \in G^2 : ts^{-1} \in \Gamma \right\}.
\]

Fix any $(s,t) \in \Gamma'$.
For each $i,j$, Markov's inequality implies that there is a subset of $Y$ of measure at least $1-\sqrt{\eta}$ on which $\mathscr{D}_y(T^{s^{-1}}P_i, T^{t^{-1}}P_j) < \sqrt{\eta}$.
Let $Y_{s,t}$ be the intersection of those sets over $0 \leq i,j \leq 1$; then we have $\nu(Y_{s,t}) \geq 1-4\sqrt{\eta}$ and $\mathscr{D}_y(T^{s^{-1}}P, T^{t^{-1}}P) < \sqrt{\eta}$ for all $y \in Y_{s,t}$.

Finally, we estimate the size of $\mathcal{I}_y$ for most $y$ (recall the statement of \Cref{lemma: IndependentPartition} for the definition of $\mathcal{I}_y$).
We have
\begin{align*}
    \int \abs{ \mathcal{I}_y \cap F_n^2 } \,d\nu(y) \ &= \ \int \sum_{(s,t) \in F_n^2} 1_{\mathscr{D}_y(T^{s^{-1}}P, T^{t^{-1}}P) \leq \sqrt{\eta}} \quad d\nu(y) \\
    &= \ \sum_{(s,t) \in F_n^2} \nu \left\{ y \in Y : \mathscr{D}_y(T^{s^{-1}}P, T^{t^{-1}}P) \leq \sqrt{\eta} \right\} \\
    &\geq \ \sum_{(s,t) \in \Gamma' \cap F_n^2} \nu(Y_{s,t}) \\
    &\geq |\Gamma' \cap F_n^2| \cdot (1-4\sqrt{\eta}).
\end{align*}
By \Cref{lemma: DiagonalSetDensity}, $\Gamma'$ has density $1$ with respect to $(F_n^2)$, so for $n$ sufficiently large we have
\[
\int \frac{\abs{ \mathcal{I}_y \cap F_n^2 }}{\abs{F_n^2}} \,d\nu(y) \ \geq \ \frac{|\Gamma' \cap F_n^2| \cdot (1-4\sqrt\eta)}{|F_n^2|} \ \geq \ 1-5\sqrt{\eta}.
\]
It follows by Markov's inequality that there is a set $Y_n^\dagger \subseteq Y$ with $\nu(Y_n^\dagger) \geq 1-\sqrt{5\sqrt{\eta}} \geq 1 - 3\eta^{1/4}$ such that
\[
\frac{\abs{\mathcal{I}_y \cap F_n^2}}{\abs{F_n^2}} \ \geq \ 1 - \sqrt{5 \sqrt{\eta}} \ \geq \ 1 - 3 \eta^{1/4}
\]
for all $y \in Y_n^\dagger$, as claimed.
\end{proof}

\begin{proof}[Proof of \Cref{proposition: BoundedImpliesIsometric}]
We show that there exists $y \in Y$ such that $\mathcal{C}_y \cap \mathcal{I}_y \neq \emptyset$.
This is sufficient because $(s,t) \in \mathcal{C}_y$ implies that $\dist_y(T^{s^{-1}}P, T^{t^{-1}}P) \leq \epsilon$, while $(s,t) \in \mathcal{I}_y$ implies that $\mathscr{D}_y(T^{s^{-1}}P, T^{t^{-1}}P) \leq \sqrt{\eta}$.
But because $\mu_y(P_0), \mu_y(P_1) \geq 1/3$ for all $y$ and $\eta < \epsilon < 10^{-6}$, these two conditions contradict each other.
Indeed, $\dist_y(T^{s^{-1}}P, T^{t^{-1}}P) \leq \epsilon$ implies in particular that 
\beq \label{equation: PartitionDistanceConsequence}
    \mu_y(T^{s^{-1}}P_0 \cap T^{t^{-1}}P_1) < \epsilon.
\eeq
But the dependence score condition implies that 
\[
    \mu_y(T^{s^{-1}}P_0 \cap T^{t^{-1}}P_1) \ > \ \mu_y(T^{s^{-1}}P_0) \mu_y(T^{t^{-1}}P_1) - \sqrt{\eta} \ > \ \frac19 - \sqrt{\eta},
\]
which contradicts \eqref{equation: PartitionDistanceConsequence}.

To find such a $y$, first choose $n$ large enough to satisfy the hypothesis of \Cref{lemma: IndependentPartition}.
Then, let $Y_n$ be the set guaranteed by \Cref{lemma: CloseTogetherPartition} and $Y_n^\dagger$ be the set guaranteed by \Cref{lemma: IndependentPartition}.
We have chosen $\epsilon$ and $\eta$ small enough to ensure $\nu(Y_n \cap Y_n^\dagger) > 0$, so choose $y \in Y_n \cap Y_n^\dagger$.
Then $|\mathcal{C}_y \cap F_n^2| \geq c(\epsilon) \cdot |F_n^2|$ and $|\mathcal{I}_y \cap F_n^2| \geq (1-3\eta^{1/4}) \cdot |F_n^2|$, so by our choice of $\eta$, we are guaranteed that $\mathcal{C}_y \cap \mathcal{I}_y \neq \emptyset$.
\end{proof}

As a corollary, we also get a characterization of weakly mixing extensions in terms of relative slow entropy.
Recall that $\pi: \X \to \Y$ is said to be \textbf{weakly mixing} if there are no intermediate isometric extensions except for the trivial one $\Y \to \Y$.

\begin{corollary} \label{corollary: WeakMixingCriterion}
    Suppose $\X$ is ergodic.
    Then $\pi: \X \to \Y$ is weakly mixing if and only if for every partition $P$ which is not 
    $\Y$-measurable,
    there exists a rate function $U$ and a F\o lner sequence $(F_n)$ such that 
    $\h^{U, (F_n)}(\X, P \,|\, \pi) > 0$.
\end{corollary}

\begin{proof}
    First suppose $\pi$ is not weakly mixing.
    Then there is 
    a nontrivial isometric extension $\mathbf{Z} \to \Y$.
    Then if $P$ is any 
    $\mathbf{Z}$-measurable
    partition, \Cref{theorem: IsometricEquivBounded} implies that 
    $\h^{U, (F_n)}(\X, P \,|\, \pi) = 0$
    for every rate function $U$ and every F\o lner sequence $(F_n)$.
    Because 
    $\mathbf{Z}$ strictly extends $\Y$,
    we can choose this $P$ to not be 
    $\Y$-measurable.

    Conversely, suppose there is a partition $P$, not measurable with respect to 
    $\Y$,
    satisfying 
    \[
    \h^{U, (F_n)}(\X, P \,|\, \pi) \ = \ 0
    \]
    for every $U$ and every $(F_n)$.
    Then the $T$-invariant $\sigma$-algebra $\pi^{-1} \mathcal{B}_Y \vee \bigvee_{s \in G} T^{s^{-1}}P$ corresponds to an intermediate extension 
    $\mathbf{Z} \to \Y$.
    Because $P$ is not 
    $\Y$-measurable,
    this is a nontrivial extension. 
    So because $P$ is relatively generating for $\mathbf{Z}$ with respect to $\Y$, this implies that 
    $\h^{U, (F_n)}(\mathbf{Z} \,|\, \pi) = 0$
    for all $U$ and $(F_n)$.
    Therefore \Cref{theorem: IsometricEquivBounded} implies that $\mathbf{Z} \to \Y$ is isometric, so $\pi$ is not weakly mixing. 
\end{proof}

\section{Rigid extensions} 
\label{sec: Rigidity}

\subsection{Definitions}

Let $G = \Z$.
Recall that a system 
$\X$
is said to be {\bf rigid} if there exists a sequence $0 = n_0 < n_1 < n_2 < \dots$ such that
\[
\lim_{k \to \infty} \mu(T^{-n_k}A \ \triangle\ A) \ = \ 0
\]
for all measurable $A \subseteq X$.

While there have been some attempts to relativize this notion and define what it means for an extension $\pi: \X \to \Y$ to be rigid (see \cite[Definition 4]{schnurr2018rigid}), thus far no definition has been completely satisfactory.
In this section we will give a new definition of rigid extension and demonstrate some of its properties.

\begin{definition} \label{definition: AutomorphismTopology}
    Let $\aut(I,m)$ denote the space of Lebesgue measure-preserving automorphisms of the unit interval $I$, modulo the equivalence relation of $m$-a.e. agreement.
    This space is a Polish topological group when endowed with the weak topology defined by the property that a sequence $(\phi_n)$ converges to $\phi$ if and only if $m(\phi_n^{-1}E \,\triangle\, \phi^{-1}E) \to 0$ for all measurable $E \subseteq I$.
    We can define a metric that generates this topology as follows.
    For $k \geq 1$, let $\D_k$ be the partition of $I$ into intervals of length $2^{-k}$.
    Then, for $\phi$, $\psi \in \aut(I,m)$, define
    \[
    d_A(\phi, \psi) \ = \ \sum_{k \geq 1} 2^{-k} \dist_m(\phi^{-1}\D_k, \psi^{-1} \D_k).
    \]
    Note in particular that a sequence $(\phi_n)$ converges to $\id \in \aut(I,m)$ if and only if 
    \[ 
    \lim_{n \to \infty} \dist_m \left( \D_j , \phi_n^{-1} \D_j \right) \ = \ 0
    \]
    for any fixed $j$.
\end{definition}

In this section, we use a different definition for the word ``cocycle'' -- we say that a \textbf{cocycle} on $Y$ is a measurable map $\a: Y \to \aut(I,m)$.
Recall that a cocycle $\a$ on $Y$ induces the skew product system
\[
\X_\a \ = \ (Y \times I, \nu \times m, T_\a)
\]
where $T_\a(y,t) := (Sy, \a(y) t)$.
For $n \in \N$, define 
\[
\a_n(y) \ := \ \a(S^{n-1}y) \circ \dots \circ \a(Sy) \circ \a(y),
\]
so that $T_\a^n(y,t) = (S^n y, \a_n(y)t)$.

\begin{definition} 
\label{definition: RigidExtension}
    We say that $\X_\a$ is a \textbf{rigid extension} of $\Y$ if for $\nu$-a.e. $y \in Y$, there is a subsequence $(n_k)$ such that $\a_{n_k}(y) \to \id$ as $k \to \infty$.
    Such a sequence $(n_k)$ is called a \textbf{rigidity sequence} for $y$.
    We will also use the terminology that $\a$ is a \textbf{rigid cocycle}.
\end{definition}

\begin{remark}
    There are a few things to note about this definition.
    \begin{enumerate}
        \item The rigidity sequence $(n_k)$ is allowed to depend on the base point $y$.
        This is the main difference between our definition and previous definitions and it is crucial to everything we are able to prove about rigid extensions.

        \item If $\Y$ is trivial, this definition reduces to the usual definition of a rigid system.
        
        \item By Rokhlin's skew product theorem, any infinite-to-one ergodic extension of $\Y$ is isomorphic to a skew product of the above form.
        However, we are not able to show that this definition of rigid extension is isomorphism invariant, so we must assume for now that $\X$ is literally a skew product over $\Y$.
    \end{enumerate}
\end{remark}

For future convenience, we record here a simple condition that implies the rigidity of a cocycle $\a$.

\begin{lemma}
\label{lemma: RigidityDiagonalization}
    Let $\a$ be a cocycle on $Y$.
    Let $\D_k$ be the depth-$k$ dyadic partition of $I$ and let $P_k = \{ Y \times E : E \in \D_k \}$.
    In order to show that $\a$ is a rigid cocycle, it is sufficient to show that for every $k \geq 1$ and every $\epsilon > 0$, we have
    \[
        \nu \left\{ y \in Y : \dist_y (P_k, T_\a^{-n}P_k) < \epsilon \text{ for infinitely many $n$} \right\} \ = \ 1.
    \]   
\end{lemma}

\begin{proof}
    \renewcommand{\R}{\mathcal{R}}
    For each $k$, define
    \begin{align*}
        \R_k \ &:= \ \left\{ y \in Y : \dist_y(P_k, T_\a^{-n}P_k) < 1/k \text{ for infinitely many $n$} \right\} \quad \text{and} \\
        \R \ &:= \ \bigcap_{k \geq 1} \R_k.
    \end{align*}
    By assumption, we have $\nu(\R) = 1$.
    We show that any $y \in \R$ has a rigidity sequence.

    Fix $y \in \R$.
    For each $k$, pick a sequence of times $n_{k, 1} < n_{k, 2} < \dots$ such that $\dist_{y}(P_k, T_\a^{-n_{k, \ell}}P_k) < 1/k$ for every $k$ and every $\ell$.
    Such a sequence exists by the definition of the sets $\mathcal{R}_k$.
    By simply deleting finitely many times if necessary, we may also assume that $n_{k+1, k+1} > n_{k,k}$ for every $k$.

    Now we claim that the sequence of times $(n_{k,k})$ is a rigidity sequence for $y$.
    By the discussion in \Cref{definition: AutomorphismTopology}, it suffices to show that for any fixed $j$,
    \[
        \dist_m(\D_j, \a_{n_{k,k}}(y)^{-1} \D_j) \to 0 \qquad \text{as $k \to \infty$}.
    \]
    Observe that for any $s \in \N$,
    \begin{align*}
        \dist_m(\D_j, \a_{s}(y)^{-1} \D_j) \ &= \ m \{ t : \D_j(t) \neq \D_j(\a_s(y)t) \} \\
        &= \ \mu_y \{ (y,t) : \D_j(t) \neq \D_j(\a_s(y)t) \} \\
        &= \ \mu_y \{ (y,t) : P_j(y,t) \neq P_j(T_\a^s(y,t)) \} \\
        &= \ \dist_y (P_j, T_\a^{-s}P_j).
    \end{align*}  
    Because $P_j$ is refined by $P_k$ for all $k \geq j$, we get
    \begin{align*}
        \lim_{k \to \infty} \dist_m(\D_j, \a_{n_{k,k}}(y)^{-1} \D_j) \ &= \ \lim_{k \to \infty} \dist_y (P_j, T_\a^{-n_{k,k}}P_j) \\
        &\leq \ \lim_{k \to \infty} \dist_y (P_k, T_\a^{-n_{k,k}}P_k) \ = \ 0
    \end{align*}
    as desired.
\end{proof}

\subsection{Genericity}

First, we show that generic extensions of an ergodic system are rigid.
First, let us recall some basic definitions.
We denote by $\co(Y)$ the set of all cocycles on $Y$, i.e. the set of all measurable maps $\a: Y \to \aut(I,m)$.
Each $\a \in \co(Y)$ induces an extension $\X_\a = (Y \times I, \nu \times m, T_\a)$ of 
$\Y$
via the skew product transformation $T_\a(y,t) = (Sy, \a(y)t)$.
By identifying each $\a \in \co(Y)$ with the skew product $T_\a \in \aut(Y \times I, \nu \times m)$, we endow $\co(Y)$ with the topology it inherits as a subspace of the weak topology on $\aut(Y \times I, \nu \times m)$.

\begin{theorem} \label{theorem: RigidGeneric}
    Let 
    $\Y$ 
    be ergodic.  Then the set of rigid cocycles $\a \in \co(Y)$ is a dense $G_\delta$ set.
\end{theorem}

\begin{remark}
    In \cite{schnurr2018rigid}, the author uses a different definition of rigid extension and shows that the set of rigid cocycles forms a $G_\delta$ set, but is not able to show that it is dense.
    Here, because we allow the rigidity sequence to depend on the base point, we are able to establish density as well.
\end{remark}

As in the previous section, let $\D_k$ be the level-$k$ dyadic partition of $I$ and let 
\[
    P_k \ = \ \{Y \times E : E \in \D_k \}.
\]
Given $k, N \in \N$, $\epsilon > 0$, and $\a \in \co(Y)$, define the set
\[
\mathcal{R}_{k,N,\epsilon}(\a) \ := \ \left\{ y \in Y : \text{ there exists an $n > N$ such that } \dist_y(P_k, T_\a^{-n} P_k) < \epsilon \right\}.
\]
Given another parameter $\eta > 0$, also define
\[
\mathcal{U}_{k,N,\epsilon, \eta} \ := \ \left\{ \a \in \co(Y) : \nu(\mathcal{R}_{k,N,\epsilon}(\a)) > 1-\eta \right\}.
\]

\begin{lemma} \label{lemma: RigidSetAsIntersection}
    The set of rigid cocycles $\a \in \co(Y)$ is given by 
    \[
    \bigcap_{k \geq 1} \bigcap_{\epsilon \searrow 0} \bigcap_{\eta \searrow 0} \bigcap_{N \geq 1} \mathcal{U}_{k,N,\epsilon, \eta},
    \]
    where the intersections over $\eta$ and $\epsilon$ should be interpreted as intersections over countable sequences tending to $0$.
\end{lemma}
\begin{proof}
    It's clear that every rigid cocycle $\a$ satisfies $\nu(\mathcal{R}_{k,N,\epsilon}(\a)) = 1$ for all $k, N, \epsilon, \eta$, so therefore $\a$ is an element of every $\mathcal{U}_{k,N,\epsilon, \eta}$.

    Conversely, suppose $\a$ is an element of every $\mathcal{U}_{k,N,\epsilon, \eta}$.
    This implies that 
    \[
    \nu(\mathcal{R}_{k,N,\epsilon}(\a)) \ > \ 1-\eta
    \] 
    for all $\eta > 0$, so $\nu(\mathcal{R}_{k,N,\epsilon}(\a)) = 1$.
    This holds for every $N$, so
    \[
    \nu \left( \bigcap_{N \geq 1} \mathcal{R}_{k,N,\epsilon}(\a) \right) \ = \ \nu \{y \in Y : \dist_y(P_k, T_\a^{-n} P_k) < \epsilon \text{ for infinitely many $n$}\} \ = \ 1
    \]
    as well.
    Finally, this holds for every $k$ and every $\epsilon > 0$, so by \Cref{lemma: RigidityDiagonalization}, we conclude that $\a$ is a rigid cocycle.
\end{proof}

\begin{lemma} \label{lemma: RigidDense}
    Each $\mathcal{U}_{k,N,\epsilon, \eta}$ is dense in $\co(Y)$.
\end{lemma}
\begin{proof}
    Recall that a \textbf{dyadic permutation of rank $M$} is an element $\phi \in \aut(I,m)$ that permutes the cells of $\D_M$ and acts as a translation on each cell.
    By \cite[Lemma 1.2]{glasner2019relativeweakmixing}, the set of piecewise constant cocycles is dense in $\co(Y)$.
    By Halmos's Weak Approximation Theorem \cite[page 65]{halmos1956ergodic}, the set of dyadic permutations is dense in $\aut(I,m)$.
    Therefore, we consider the dense set $\mathscr{D}$ of cocycles $\a$ such that $\{ \a(y) : y \in Y \}$ is a finite set of dyadic permutations.
    We show that each of the sets $\mathcal{U}_{k,N,\epsilon, \eta}$ contains $\mathscr{D}$.
    To do this, it is clearly sufficient to show that each $\a \in \mathscr{D}$ is a rigid cocycle.

    Fix $\a \in \mathscr{D}$.
    Because $\a$ takes only finitely many values, there is some $M$ such that each $\a(y)$ is a dyadic permutation of rank $M$.
    So we may consider $\a$ to be a map from $Y$ into $\sym_M$ (the subgroup of $\aut(I,m)$ consisting of dyadic permutations of rank $M$, isomorphic to the symmetric group on $M$ elements).
    Define
    \[
        \mathcal{R} \ = \ \{ y \in Y : \a_n(y) = \id \text{ for infinitely many $n$} \}.
    \]  
    We want to show that $\nu(\mathcal{R}) = 1$.
    To do this, fix $y \in Y$ and let
    \[
        \Sigma_y \ = \ \{ \sigma \in \sym_M : \a_n(y) = \sigma \text{ for infinitely many $n$} \}.
    \]  
    Observe that $\{n \in \N : \a_n(y) \in \Sigma_y \}$ must be co-finite.
    Now we claim that if $\a_n(y) \in \Sigma_y$, then $S^n y \in \mathcal{R}$.
    This is because if $\a_n(y) \in \Sigma_y$, then there are infinitely many $m > n$ satisfying $\a_m(y) = \a_n(y)$.
    For all such $m$, we have
    $\a_{m-n}(S^n y) \a_n(y) = \a_m(y)$, which implies that $\a_{m-n}(S^n y) = \id$ for infinitely many $m$, so $S^n y \in \mathcal{R}$.

    Therefore we have shown that for every $y \in Y$, the set of $n$ such that $S^n y \in \mathcal{R}$ is co-finite.
    By ergodicity, this implies that $\nu(\mathcal{R}) = 1$ as desired. 
\end{proof}

\begin{lemma} \label{lemma: RigidOpen}
    Each $\mathcal{U}_{k,N,\epsilon, \eta}$ is open in $\co(Y)$.
\end{lemma}
\begin{proof}
    Fix $\a \in \mathcal{U}_{k,N,\epsilon, \eta}$.
    We may write the set $\mathcal{R}_{k,N,\epsilon}(\alpha)$ as
    \begin{align*}
    &\mathcal{R}_{k,N,\epsilon}(\alpha) = \\
    &\ \bigcup_{M > N} \bigcup_{\epsilon' < \epsilon} \{y \in Y : \text{ there exists some $n \in (N,M]$ such that } \dist_y(P_k, T_\a^{-n} P_k) < \epsilon' \}.
    \end{align*}
    Since $\nu(\mathcal{R}_{k,N,\epsilon}) > 1-\eta$, it follows that there exist $M > N$, $\epsilon' < \epsilon$, and $\eta' < \eta$ such that 
    \begin{align*}
    &\nu(\mathcal{R}') \ := \\
    &\nu \{y \in Y : \text{ there exists some $n \in (N,M]$ such that } \dist_y(P_k, T_\a^{-n} P_k) < \epsilon' \} = 1-\eta'.
    \end{align*}

    Let $\sigma > 0$ be a parameter that is so small that $\sigma < \epsilon - \epsilon'$ and $(M-N)\sigma < \eta - \eta'$.
    Let $\mathcal{O}$ be an open neighborhood of $\alpha$ that is so small that for any $\b \in \mathcal{O}$, we have 
    \[
    \dist_\mu(T_\a^{-n} P_k, T_\b^{-n} P_k) \ < \ \sigma^2 
    \] 
    for all $n \in (N,M]$.
    This is possible because 
    \begin{itemize}
        \item for every $n$, the map $T_\b \mapsto T_\b^n$ is a continuous map from $\aut(Y \times I, \nu \times m)$ to itself because $\aut(Y \times I, \nu \times m)$ is a topological group, and
        \item for any fixed $\gamma \in \co(Y)$, the map $\b \mapsto \dist_\mu(T_\gamma^{-1} P_k, T_\b^{-1} P_k)$ is a continuous map $\co(Y) \to [0,1]$ by definition of the weak topology on $\aut(Y \times I, \nu \times m)$. 
    \end{itemize}
    We show that any $\b \in \mathcal{O}$ is also in $\mathcal{U}_{k,N,\epsilon, \eta}$.

    Because $\dist_\mu(P,Q) = \int \dist_y(P,Q) \,d\nu(y)$, we apply Markov's inequality to conclude that for each $n \in (N,M]$, there is a set of measure $>1-\sigma$ on which $\dist_y(T_\a^{-n} P_k, T_\b^{-n}P_k) < \sigma$.
    Now define
    \[
    \wt{Y} \ := \ \{y \in Y : \dist_y(T_\a^{-n} P_k, T_\b^{-n}P_k) < \sigma \text{ for all $n \in (N,M]$} \}
    \]
    and note that $\nu(\wt{Y}) > 1-(M-N)\sigma$.

    Now consider some $y \in \wt{Y} \cap \mathcal{R}'$.
    Because $y \in \mathcal{R}'$, there is $n \in (N,M]$ such that $\dist_y(P_k, T_\a^{-n}P_k) < \epsilon'$.
    Then, for that same $n$, we get the estimate
    \[
        \dist_y(P_k, T_\b^{-n}P_k) \ \leq \ \dist_y(P_k, T_\a^{-n}P_k) + \dist_y(T_\a^{-n} P_k, T_\b^{-n}P_k) \ < \ \epsilon' + \sigma \ < \ \epsilon,
    \]
    where the second inequality holds because $y \in \wt{Y}$.

    This shows that for all $y \in \wt{Y} \cap \mathcal{R}'$, there exists an $n \in (N,M]$ such that $\dist_y(P_k, T_\b^{-n}P_k) < \epsilon$, showing that $\mathcal{R}_{k,N,\epsilon}(\beta) \supseteq \wt{Y} \cap \mathcal{R}'$.
    Since $\nu(\wt{Y} \cap \mathcal{R}') > 1 - (M-N)\sigma - \eta' > 1-\eta$, it follows that $\beta \in \mathcal{U}_{k,N,\epsilon, \eta}$ as desired.
\end{proof}

\begin{proof}[Proof of \Cref{theorem: RigidGeneric}]
    Follows immediately from \Cref{lemma: RigidSetAsIntersection,lemma: RigidDense,lemma: RigidOpen} and the Baire category theorem.
\end{proof}

\subsection{Relationship between rigidity and slow entropy}

Throughout this section, let $L$ denote the rate function $L(n) = \log n$.
First we give a sufficient condition for an extension to be rigid.

\begin{theorem} \label{theorem: RelativeRigidtySufficientCondition}
    Assume that $\Y$ is ergodic.
    Let $\a$ be a cocycle on $Y$ and let $\pi: \X_\a \to \Y$ denote projection onto the first coordinate.
    Suppose that there exists a F\o lner sequence $(F_n)$ for $\N$ such that $\h^{L, (F_n)}(\X_\a \,|\, \pi) = 0$.
    Then $\X_\a$ is a rigid extension of $\Y$. 
\end{theorem}

\begin{proof}
    For a partition $P$, $\epsilon > 0$, and $m \in \N$, define
    \[
        \mathcal{R}_{P,\epsilon,m} \ = \ \left\{ y \in Y : \text{ there exists $k>m$ such that } \dist_y(P, T_\a^{-k}P) < 5\sqrt\epsilon \right\}.
    \]
    The first step is to show that $\nu(\mathcal{R}_{P,\epsilon,m}) \geq 1-4\sqrt\epsilon$ for every $P, \epsilon, m$.

    Let $(F_n)$ be the F\o lner sequence given by the hypothesis of \Cref{theorem: RelativeRigidtySufficientCondition}.
    By the assumption that $\h^{L,(F_n)}(\X \,|\, \pi) = 0$, for all $n$ sufficiently large we have
    \beq \label{equation: LogDominated}
        |P|^{\cov(\mu, P, F_n, \epsilon \,|\, \pi)} \ \leq \ \frac{\epsilon}{m} \cdot |F_n|.  
    \eeq
    Also, by the mean ergodic theorem, for all sufficiently large $n$ we have
    \beq \label{equation: ErgodicityInBase}
        \nu \left\{ y \in Y : \frac{\abs{ \{t \in F_n : S^t y \in \mathcal{R}_{P,\epsilon,m} \} }}{|F_n|} \ < \ \nu(\mathcal{R}_{P,\epsilon,m}) + \epsilon \right\} \ \geq \ 1-\epsilon.
    \eeq
    So fix an $n$ which is large enough so that \eqref{equation: LogDominated} and \eqref{equation: ErgodicityInBase} both hold.

    Let 
    \[
    C \ = \ C(n) = \cov(\mu, P, F_n, \epsilon \,|\, \pi)
    \] 
    and let $Y_n \subseteq Y$ be the set of $y$ satisfying $\cov(\mu_y, P, F_n, \epsilon) \leq C$.
    By definition we have $\nu(Y_n) \geq 1-\epsilon$, so by \eqref{equation: ErgodicityInBase}, we may fix a point $y$ that is an element of both $Y_n$ and the set appearing in \eqref{equation: ErgodicityInBase}.

    We now repeat the construction from the proof of \Cref{lemma: CloseTogetherPartition}, which we partially reproduce here for convenience.
    Let $B_1, \dots, B_L$ be subsets of $X$ such that each $B_i$ has $d_{P,F_n}$-diameter at most $\epsilon$ and $\mu_y\left( \bigcup B_i \right) \geq 1 - \epsilon$.
    Let $X' = \bigcup B_i$.
    Without loss of generality, we may assume that the $B_i$ are disjoint.
    For each $i$, fix a point $x_i \in B_i$.
    Then, for $x \in X'$, define $r(x)$ to be the unique $x_i$ such that $x \in B_i$.

    Define
    \[
    A \ = \ \Big\{ s \in F_n : \mu_y \{x \in X' : P(T_\a^s x) = P(T_\a^s r(x)) \} \geq 1-\sqrt{2\epsilon} \Big\}.
    \]
    In the proof of \Cref{lemma: CloseTogetherPartition}, we proved the estimate 
    \[
    |A| \ \geq \ |F_n|(1-\sqrt{2\epsilon}).
    \]
    Now decompose the set $A$ as
    \[
    A \ = \ \bigcup_{w \in  \{0,1,\dots,|P|-1\}^C} \{s \in A : (P(T_\a^s x_i))_{i=1}^L = w \} \ =: \ \bigcup_{w \in  \{0,1,\dots,|P|-1\}^C} A_w.
    \]
    In the proof of \Cref{lemma: CloseTogetherPartition}, we also showed that 
    \beq
    \label{eq: CloseTogetherTimes}
        \dist_y(T_\a^{-s}P, T_\a^{-t}P) \ \leq \ 5\sqrt{\epsilon} \qquad \text{whenever $s,t$ lie in the same $A_w$}.
    \eeq

    Using this decomposition of $A$ into the sets $A_w$, we can show that $S^t y \in \mathcal{R}_{P,\epsilon,m}$ for most $t \in F_n$.
    By \eqref{eq: CloseTogetherTimes}, we conclude that $\{t \in F_n : S^t y \in \mathcal{R}_{P,\epsilon,m} \}$ contains all of the elements of $A$, except for possibly the $m$ largest elements of each $A_w$.
    Therefore, we can use \eqref{equation: LogDominated} to estimate
    \begin{align*}
        \#\{t \in F_n : S^t y \in \mathcal{R}_{P,\epsilon, m} \} \ &\geq \ \sum_{w \in  \{0,1,\dots,|P|-1\}^C} \left( |A_w| - m \right) \ = \ |A| - m \cdot |P|^C \\
        &\geq \ |F_n|(1-\sqrt{2\epsilon}) - m \cdot \frac{\epsilon}{m} \cdot |F_n| \\
        &\geq \ |F_n| (1 - 3\sqrt\epsilon).
    \end{align*}
    Combining this estimate with \eqref{equation: ErgodicityInBase}, we conclude that
    \[
        \nu(\mathcal{R}_{P,\epsilon, m}) \ \geq \ 1-3\sqrt\epsilon - \epsilon \ \geq \ 1-4\sqrt\epsilon
    \]      
    as desired.

    Now let
    \[    
        \mathcal{R}_{P,\epsilon} \ = \ \bigcap_{M \geq 1} \bigcup_{m \geq M} \mathcal{R}_{P,\epsilon,m} \ = \ \left\{y \in Y: \dist_y(P, T_\a^{-n}P) < 5\sqrt\epsilon \text{ for infinitely many $n$} \right\}.
    \]  
    Because each $\nu(\mathcal{R}_{P,\epsilon,m}) \geq 1-4\sqrt{\epsilon}$, we have $\nu(\mathcal{R}_{P,\epsilon}) \geq 1-4\sqrt{\epsilon}$ as well.

    Finally, let $P_k$ be the partition $\{Y \times E : E \in  \D_k\}$, let $\epsilon_k = 1/k$, and let $\mathcal{R}_k = \mathcal{R}_{P_k, \epsilon_k}$.
    We have $\nu(\mathcal{R}_k) \geq 1-4/\sqrt{k}$.
    Now let $\bar{\mathcal{R}} = \bigcap_{K \geq 1} \bigcup_{k \geq K} \mathcal{R}_k$ and note that
    \[
        \nu\left( \bar{\mathcal{R}} \right) \ = \ \lim_{K \to \infty} \nu \left( \bigcup_{k \geq K} \mathcal{R}_k \right) \ \geq \ \lim_{K \to \infty} \nu(\mathcal{R}_K) \ = \ 1.
    \]  

    We claim that every $y \in \bar{\mathcal{R}}$ has a rigidity sequence.

    Fix $y \in \bar{\mathcal{R}}$.
    Then, by construction, there are infinitely many $k$ that satisfy 
    \[
        \dist_y(P_k, T_\a^{-n}P_k) < 5/\sqrt{k} \quad \text{for infinitely many $n$}.
    \]
    So, by repeating the diagonalization argument from \Cref{lemma: RigidityDiagonalization}, we again are able to conclude that $y$ has a rigidity sequence.
\end{proof}

\begin{corollary}
    Isometric extensions are rigid.
\end{corollary}
\begin{proof}
    Suppose $\pi: \X_\a \to \Y$ is an isometric extension.
    Then by \Cref{theorem: IsometricEquivBounded}, for any F\o lner sequence and any rate function $U$, $\pi$ has zero relative slow entropy.
    So in particular, there is a F\o lner sequence for which $\pi$ has zero relative slow entropy with respect to the rate function $L(n) = \log n$.
    By \Cref{theorem: RelativeRigidtySufficientCondition}, this implies $\pi$ is rigid.
\end{proof}

\begin{remark}
    In the non-relative setting, this result can be proven directly from the definitions using the fact that any orbit of a compact group rotation is dense in some closed subgroup.
    In the relative setting, it can proven in a similar but more complicated way by appealing to the theory of the Mackey group.
    It is interesting to note that we are able to provide another proof of this result using entropy methods. 
\end{remark}

In the non-relative setting, we are also able to prove a converse and obtain necessary and sufficient conditions for rigidity in terms of slow entropy.
For this part, we use interchangeably the notations $f(m) \ll g(m)$ and $f(m) = o(g(m))$ to mean that $f(m)/g(m) \to 0$ as $m \to \infty$.

\begin{theorem} \label{theorem: RigidtyCondition}
The following are equivalent.
\begin{enumerate}
    \item 
    $\X$
    is rigid.
    \item For every rate function $U$, there exists a F\o lner sequence $(F_n)$ for $\N$ such that 
    $\h^{U, (F_n)}(\X) = 0$.
    \item There exists a F\o lner sequence $(F_n)$ for $\N$ such that 
    $\h^{L, (F_n)}(\X) = 0$.
\end{enumerate} 
\end{theorem}

\begin{proof}
The implication $(2) \Longrightarrow (3)$ is trivial and the implication $(3) \Longrightarrow (1)$ is the special case of \Cref{theorem: RelativeRigidtySufficientCondition} where $\Y$ is trivial, so we just need to show that $(1) \Longrightarrow (2)$.
Assume that 
$\X$
is rigid and let $(n_k)$ be a rigidity sequence.
Let $P$ be a finite generating partition for 
$\X$.
Such a partition must exist by Krieger's theorem \cite{krieger1970entropy} because rigid systems have zero entropy.
Applying the definition of rigidity to each of the finitely many cells of $P$, it follows that
\[
\lim_{k \to \infty} \dist_\mu(T^{-n_k}P, P) \ = \ 0.
\] 
Now replace the rigidity sequence $(n_k)$ with a sufficiently thin subsequence so that we may assume that 
\beq \label{eq: PartitionDistanceDecayRate}
\dist_\mu(T^{-n_k}P, P) \ < \ 2^{-k}.
\eeq
Also assume that the rigidity sequence is sparse enough so that $n_{k+1} - n_k > k$ for every $k$.

Now let $U$ be an arbitrary rate function and assume without loss of generality that $U(m) \ll \exp(m)$.
Let $V$ be another rate function satisfying $V(m) \ll \log U(m)$.
We define our F\o lner sequence $(F_m)$ by the formula 
\[
F_m \ := \ [0, V(m)) \cup [n_1, n_1 + V(m)) \cup \dots \cup [n_{m-1}, n_{m-1} + V(m)).
\]
It's clear that this is a F\o lner sequence for $\N$.

It will be useful later to have a good estimate for $|F_m|$.  Clearly $|F_m| \leq m\cdot V(m)$, but we can also show that it is not much smaller than this.  Observe that
\begin{align}
|F_m| \ &\geq \ V(m) \cdot \#\{k \leq m : [n_k, n_k + V(m)) \cap [n_{k+1}, n_{k+1} + V(m)) = \emptyset \} \nonumber \\
&\geq \ V(m) \cdot \#\{k \leq m: n_{k+1} - n_k > V(m) \} \nonumber \\
&\geq \ V(m) \cdot \#\{k \leq m: k > V(m) \} \nonumber \\
&= \ V(m) \cdot \max(m-V(m), 0) \nonumber \\
&= \ V(m) \cdot (m - o(m)) \label{eq: FolnerSetSize}
\end{align}

Our goal is to show that $\h^{U, (F_m)}(\X) = 0$.
Since $P$ is a generating partition, it suffices to show that $\h^{U, (F_m)}(\X, P) = 0$.
To do this, let $\epsilon > 0$.
We seek to estimate $\cov(\mu, P, F_m, \epsilon)$ for $m$ sufficiently large.
Let $C = \cov(\mu, P, [0,V(m)), \epsilon)$ and let $B_1, \dots, B_C$ be subsets of $X$ satisfying $\mu \left( \bigcup B_i \right) \geq 1-\epsilon$ and 
\[
\diam_{P, [0, V(m))}(B_i) \ \leq \ \epsilon.
\]
We now show that we can restrict the $B_i$ to a large subset of $X$ such that after the restsriction, $\diam_{P, F_m}(B_i)$ is also small.  

Let $k_0 = k_0(m)$ be the smallest integer that satisfies
\[
\sum_{k \geq k_0} \dist_\mu(T^{-n_k}P, P) \ < \ \frac{\epsilon}{V(m)}.
\]
Because of the condition that $\dist_{\mu}(T^{-n_k}P, P) < 2^{-k}$, it follows that 
\beq \label{eq: SizeOfK0}
k_0(m) \ \leq \ \log_2 \left( \frac{V(m)}{\epsilon} \right) \ \ll \ \log \log U(m) \ \ll \ m.
\eeq
For $0 \leq i < V(m)$, define the ``good sets'' 
\begin{align}
\mathcal{G}_i \ &:= \ \{x \in X : P(T^i x) = P(T^{n_k + i}x) \text{ for all $k \geq k_0$} \} \quad \text{and} \\
\mathcal{G} \ &:= \ \bigcap_{i=0}^{V(m)} \mathcal{G}_i. 
\end{align}
By the definition of $k_0$ and the $T$-invariance of $\mu$, we have
\beq \label{eq: RigidityUsingInvariance}
\mu(\mathcal{G}^c) \ \leq \ \sum_{i=0}^{V(m)-1} \mu(\mathcal{G}_i^c) \ \leq \ \sum_{i=0}^{V(m)-1} \sum_{k \geq k_0} \dist_\mu(T^{-(n_k+i)}P, T^{-i}P) \ \leq \ \epsilon.
\eeq
Now replace each $B_i$ by $B_i' = B_i \cap \mathcal{G}$, so we still have $\mu \left( \bigcup B_i' \right) \geq 1-2\epsilon$.
It remains to show that each $B_i'$ has small diameter according to $d_{P, F_m}$.

If $x,y \in B_i'$, then
\begin{align*}
    |F_m| \cdot d_{P, F_m}(x,y) \ &\leq \ \sum_{k=0}^{m-1} \sum_{i=0}^{V(m)-1} 1_{P(T^{n_k + i}x) \neq P(T^{n_k + i}y)} \\
    &\leq \ k_0 \cdot V(m) + \sum_{k = k_0}^{m-1} \sum_{i=0}^{V(m)-1} 1_{P(T^{n_k + i}x) \neq P(T^{n_k + i}y)} \\
    &= \ k_0 \cdot V(m) + (m - k_0) \cdot \sum_{i=0}^{V(m)-1} 1_{P(T^{i}x) \neq P(T^{i}y)} \\
    &\leq \ k_0 \cdot V(m) + m \cdot V(m) \cdot d_{P, [0, V(m))}(x,y).
\end{align*}
Therefore, by \eqref{eq: FolnerSetSize} and \eqref{eq: SizeOfK0}, we have
\begin{align*}
\diam_{P, F_m}(B_i') \ &\leq \ \frac{k_0 \cdot V(m)}{|F_m|} + \frac{m \cdot V(m) \cdot \diam_{P, [0, V(m))}(B_i')}{|F_m|} \\
&\leq \frac{o(m)}{m-o(m)} + \frac{\epsilon \cdot m}{m - o(m)} \\
&\leq 3 \epsilon
\end{align*}
for sufficiently large $m$.

Thus we have shown that 
\[
\cov(\mu, P, F_m, 3\epsilon) \ \leq \ \cov(\mu, P, [0, V(m)), \epsilon) \ \leq \ |P|^{V(m)} \ \ll \ U(m) \ \ll \ U(|F_m|)
\]
for any $\epsilon > 0$, and the desired conclusion follows.
\end{proof}

\begin{remark}
    The part of this proof that breaks down in the relative setting is the estimate \eqref{eq: RigidityUsingInvariance}.
    Here we have used the $T$-invariance of $\mu$ critically to deduce that if the partitions $P$ and $T^{-n_k}P$ are close with respect to $\mu$, then so are $T^{-i}P$ and $T^{-(n_k + i)}P$.
    In the relative setting this breaks down because if $P$ and $T_\a^{-n_k}P$ are close with respect to $\mu_y$, then $T_\a^{-i}P$ and $T_\a^{-(n_k + i)}P$ are only close with respect to $\mu_{S^i y}$.
\end{remark}

\begin{remark}
    In \cite[Theorem 1]{adams2021generic}, the author shows that for the F\o lner sequence $F_n = [0,n)$ and any sub-exponential rate function $U$, there is a dense $G_\delta$ set of systems $\mathbf{I} = ([0,1], T, m)$ that are both rigid and satisfy $\h^{U, (F_n)}(\mathbf{I}) = \infty$.
    Combined with \Cref{theorem: RigidtyCondition}, this shows that generically, the slow entropy of a system depends quite strongly on the choice of F\o lner sequence.
    This is in contrast with Kolmogorov--Sinai entropy, which is independent of the choice of F\o lner sequence.   
\end{remark}

As a corollary of \Cref{theorem: RigidtyCondition}, we get a similar condition that characterizes mild mixing systems in terms of slow entropy.
Recall that a system is said to be \textbf{mildly mixing} if it has no nontrivial rigid factors \cite{furstenberg1978mild}.

\begin{corollary} \label{corollary: MildMixingCondition}
The system 
$\X$
is mildly mixing if and only if for all partitions $P$ of $X$ and all F\o lner sequences $(F_n)$ for $\N$, we have $\h^{L, (F_n)}(\X, P) > 0$.
\end{corollary}

\begin{proof}
    Suppose 
    $\X$
    is not mildly mixing.
    Then there is a nontrivial rigid system $\Y$ and a factor map
    $\pi: \X \to \Y$.
    Let $Q$ be any partition of $Y$ and let $P = \pi^{-1} Q$.
    The rigidity of $\Y$ implies that we can find a F\o lner sequence $(F_n)$ such that $\h^{L, (F_n)}(\Y, Q) = 0$.
    Then, using the definition of factor map, it immediately follows that $\h^{L, (F_n)}(\X, P) = 0$ as well.

    Conversely, suppose that there exist a partition $P$ and a F\o lner sequence $(F_n)$ so that 
    \[
    \h^{L, (F_n)}(\X, P) \ = \ 0.
    \]    
    Then consider the factor 
    $\Y$
    corresponding to the $T$-invariant $\sigma$-algebra $\bigvee_{n \in \Z} T^{-n}P$.
    Because $P$ is a generating partition for this factor, it follows that $\h^{L, (F_n)}(\Y) = 0$, which implies that 
    $\Y$
    is rigid, so 
    $\X$
    is not mildly mixing.
\end{proof}

\subsection{A necessary condition for rigidity}

In light of \Cref{theorem: RigidGeneric} and our failure to prove the converse of \Cref{theorem: RelativeRigidtySufficientCondition}, one may wonder whether or not \emph{every} extension is rigid.
In this section, we show that this is not the case by exhibiting a natural non-empty class of extensions that can not be rigid.

\begin{definition}
\label{definition: MixingCocycle}
    Given a system $\Y$, a cocycle $\a$ on $Y$ is said to be \textbf{strongly mixing} if for $\nu$-a.e. $y$, we have
    \[
        m \left( E \cap \a_n(y)^{-1} E \right) \ \to \ m(E)^2
    \]
    for all measurable $E \subseteq I$.
    Also, we will say that a skew product extension $\X_\a \to \Y$ is a \textbf{strongly mixing extension} if $\a$ is a strongly mixing cocycle.
\end{definition}

This definition is similar to the definition of strongly mixing extension given in \cite{schnurr2018rigid}, but here we have phrased it to be more analogous to our definition of rigidity.
It is unknown whether or not the definitions of rigidity and strong mixing presented here are equivalent to the definitions given in \cite{schnurr2018rigid}.

\begin{proposition}
    The set of rigid cocycles and the set of strongly mixing cocycles are disjoint.
\end{proposition}

\begin{proof}
    Let $E \subseteq I$ be any subset of measure $1/2$.
    If $\a$ is a strongly mixing cocycle, then for a.e. $y$, 
    \[
        m \left( E \cap \a_n(y)^{-1}E \right) \ \to \ 1/4
    \]
    as $n \to \infty$.
    But if $\a$ were also a rigid cocycle, then there would have to be a subsequence $(n_k)$ along which 
    \[
        m \left( E \cap \a_{n_k}(y)^{-1}E \right) \ \to \ 1/2
    \] 
    as $k \to \infty$, a contradiction.
\end{proof}

Finally, let us remark that the set of strongly mixing cocycles is non-empty.
Indeed, if $\mathbf{I} = (I, m, \a_0)$ is any strongly mixing system and $\a$ is the constant cocycle $\a(y) = \a_0$, then $\a$ is clearly a strongly mixing cocycle.
These cocycles correspond to direct product transformations on $Y \times I$ where the transformation in the $I$ coordinate is strongly mixing.

\bibliography{../BibTex-refs/mybib.bib}{}

\begin{thebibliography}{AGTW21}

\bibitem[Ada21]{adams2021generic}
Terrence Adams.
\newblock Genericity and rigidity for slow entropy transformations.
\newblock {\em New York J. Math.}, 27:393--416, 2021.

\bibitem[AGTW21]{austin2021dominant}
Tim Austin, Eli Glasner, Jean-Paul Thouvenot, and Benjamin Weiss.
\newblock An ergodic system is dominant exactly when it has positive entropy.
\newblock arXiv:2112.03800, 2021.

\bibitem[AR62]{abramov1962skew}
L.~M. Abramov and V.~A. Rohlin.
\newblock Entropy of a skew product of mappings with invariant measure.
\newblock {\em Vestnik Leningrad. Univ.}, 17(7):5--13, 1962.

\bibitem[Bil95]{billingsley1995book}
Patrick Billingsley.
\newblock {\em Probability and measure}.
\newblock Wiley Series in Probability and Mathematical Statistics. John Wiley
  \& Sons, Inc., New York, third edition, 1995.
\newblock A Wiley-Interscience Publication.

\bibitem[Bow71]{bowen1971entropy}
Rufus Bowen.
\newblock Entropy for group endomorphisms and homogeneous spaces.
\newblock {\em Trans. Amer. Math. Soc.}, 153:401--414, 1971.

\bibitem[CK20]{cyr2020subshift}
Van Cyr and Bryna Kra.
\newblock Realizing ergodic properties in zero entropy subshifts.
\newblock {\em Israel J. Math.}, 240(1):119--148, 2020.

\bibitem[ELW21]{einsiedler2021entropy}
Manfred Einsiedler, Elon Lindenstrauss, and Thomas Ward.
\newblock {\em Entropy in ergodic theory and topological dynamics}.
\newblock Book draft, available online at
  \url{https://tbward0.wixsite.com/books/entropy}, 2021.

\bibitem[Fer97]{ferenczi1997complexity}
S\'{e}bastien Ferenczi.
\newblock Measure-theoretic complexity of ergodic systems.
\newblock {\em Israel J. Math.}, 100:189--207, 1997.

\bibitem[Fur77]{furstenberg1977szemeredi}
Harry Furstenberg.
\newblock Ergodic behavior of diagonal measures and a theorem of
  {S}zemer\'{e}di on arithmetic progressions.
\newblock {\em J. Analyse Math.}, 31:204--256, 1977.

\bibitem[FW78]{furstenberg1978mild}
Hillel Furstenberg and Benjamin Weiss.
\newblock The finite multipliers of infinite ergodic transformations.
\newblock In {\em The structure of attractors in dynamical systems ({P}roc.
  {C}onf., {N}orth {D}akota {S}tate {U}niv., {F}argo, {N}.{D}., 1977)}, volume
  668 of {\em Lecture Notes in Math.}, pages 127--132. Springer, Berlin, 1978.

\bibitem[Gla03]{glasner2003ergodic}
Eli Glasner.
\newblock {\em Ergodic theory via joinings}, volume 101 of {\em Mathematical
  Surveys and Monographs}.
\newblock American Mathematical Society, Providence, RI, 2003.

\bibitem[Gra11]{gray2011entropy}
Robert~M. Gray.
\newblock {\em Entropy and information theory}.
\newblock Springer, New York, second edition, 2011.

\bibitem[GW19]{glasner2019relativeweakmixing}
Eli Glasner and Benjamin Weiss.
\newblock Relative weak mixing is generic.
\newblock {\em Sci. China Math.}, 62(1):69--72, 2019.

\bibitem[Hal56]{halmos1956ergodic}
Paul~R. Halmos.
\newblock {\em Lectures on ergodic theory}, volume~3 of {\em Publications of
  the Mathematical Society of Japan}.
\newblock Mathematical Society of Japan, Tokyo, 1956.

\bibitem[Kat80]{katok1980lyapunov}
A.~Katok.
\newblock Lyapunov exponents, entropy and periodic orbits for diffeomorphisms.
\newblock {\em Inst. Hautes \'{E}tudes Sci. Publ. Math.}, (51):137--173, 1980.

\bibitem[Kie75]{kieffer1975amenable}
J.~C. Kieffer.
\newblock A generalized {S}hannon-{M}c{M}illan theorem for the action of an
  amenable group on a probability space.
\newblock {\em Ann. Probability}, 3(6):1031--1037, 1975.

\bibitem[KL16]{KerrLi2016Ergodic}
David Kerr and Hanfeng Li.
\newblock {\em Ergodic theory}.
\newblock Springer Monographs in Mathematics. Springer, Cham, 2016.
\newblock Independence and dichotomies.

\bibitem[Kol58]{kolmogorov1958entropy}
A.~N. Kolmogorov.
\newblock A new metric invariant of transient dynamical systems and
  automorphisms in {L}ebesgue spaces.
\newblock {\em Dokl. Akad. Nauk SSSR (N.S.)}, 119:861--864, 1958.

\bibitem[Kol59]{kolmogorov1959entropy}
A.~N. Kolmogorov.
\newblock Entropy per unit time as a metric invariant of automorphisms.
\newblock {\em Dokl. Akad. Nauk SSSR}, 124:754--755, 1959.

\bibitem[Kri70]{krieger1970entropy}
Wolfgang Krieger.
\newblock On entropy and generators of measure-preserving transformations.
\newblock {\em Trans. Amer. Math. Soc.}, 149:453--464, 1970.

\bibitem[KT97]{katok1997slow}
Anatole Katok and Jean-Paul Thouvenot.
\newblock Slow entropy type invariants and smooth realization of commuting
  measure-preserving transformations.
\newblock {\em Ann. Inst. H. Poincar\'{e} Probab. Statist.}, 33(3):323--338,
  1997.

\bibitem[KVW19]{kanigowski2019parabolic}
Adam Kanigowski, Kurt Vinhage, and Daren Wei.
\newblock Slow entropy of some parabolic flows.
\newblock {\em Comm. Math. Phys.}, 370(2):449--474, 2019.

\bibitem[KW72]{katznelson1972commuting}
Yitzhak Katznelson and Benjamin Weiss.
\newblock Commuting measure-preserving transformations.
\newblock {\em Israel J. Math.}, 12:161--173, 1972.

\bibitem[Lot22]{lott2022dominant}
Adam Lott.
\newblock Zero entropy actions of amenable groups are not dominant.
\newblock arXiv:2204.11459, 2022.

\bibitem[MO85]{ollagnier1985book}
Jean Moulin~Ollagnier.
\newblock {\em Ergodic theory and statistical mechanics}, volume 1115 of {\em
  Lecture Notes in Mathematics}.
\newblock Springer-Verlag, Berlin, 1985.

\bibitem[OW80]{ornstein1980rokhlin}
Donald~S. Ornstein and Benjamin Weiss.
\newblock Ergodic theory of amenable group actions. {I}. {T}he {R}ohlin lemma.
\newblock {\em Bull. Amer. Math. Soc. (N.S.)}, 2(1):161--164, 1980.

\bibitem[RW00]{rudolph2000mixing}
Daniel~J. Rudolph and Benjamin Weiss.
\newblock Entropy and mixing for amenable group actions.
\newblock {\em Ann. of Math. (2)}, 151(3):1119--1150, 2000.

\bibitem[Sch18]{schnurr2018rigid}
Mike Schnurr.
\newblock A note on strongly mixing extensions.
\newblock arXiv:1712.06192, 2018.

\bibitem[Sin59a]{sinai1959entropy1}
Ja. Sina\u{\i}.
\newblock Flows with finite entropy.
\newblock {\em Dokl. Akad. Nauk SSSR}, 125:1200--1202, 1959.

\bibitem[Sin59b]{sinai1959entropy2}
Ja. Sina\u{\i}.
\newblock On the concept of entropy for a dynamic system.
\newblock {\em Dokl. Akad. Nauk SSSR}, 124:768--771, 1959.

\bibitem[Wal82]{walters1982book}
Peter Walters.
\newblock {\em An introduction to ergodic theory}, volume~79 of {\em Graduate
  Texts in Mathematics}.
\newblock Springer-Verlag, New York-Berlin, 1982.

\bibitem[WZ92]{ward1992amenable}
Thomas Ward and Qing Zhang.
\newblock The {A}bramov-{R}okhlin entropy addition formula for amenable group
  actions.
\newblock {\em Monatsh. Math.}, 114(3-4):317--329, 1992.

\bibitem[ZK]{zorinkranichCompactIsometric}
Pavel Zorin-Kranich.
\newblock Compact extensions are isometric.
\newblock
  \href{https://www.math.uni-bonn.de/~pzorin/notes/compact-isometric.pdf}{https://www.math.uni-bonn.de/$\sim$pzorin/notes/compact-isometric.pdf}.

\end{thebibliography}
\bibliographystyle{alpha}

\end{document}